%% file: main.tex
\documentclass{article}
\usepackage[final,nonatbib]{neurips_2020}

\usepackage[utf8]{inputenc} 
\usepackage[T1]{fontenc}    
\usepackage{hyperref}       
\usepackage{url}            
\usepackage{booktabs}       
\usepackage{amsfonts, amssymb, mathtools, amsthm}       
\usepackage{nicefrac}       
\usepackage{microtype}      
\usepackage{xcolor}
\usepackage{mystyle}
\usepackage{pgfplots}
\pgfplotsset{compat=1.3}
\usepackage{caption}
\usepackage{enumitem}
\usepackage{subcaption}

\renewcommand{\epsilon}{\varepsilon}

\title{Faster Wasserstein Distance Estimation \\ with the Sinkhorn Divergence}

\author{%
  L{\'e}na{\"i}c Chizat$^1$\thanks{Corresponding author: \texttt{lenaic.chizat@universite-paris-saclay.fr}}\hspace{1.4mm}, \quad Pierre Roussillon$^2$, \quad Flavien L{\'e}ger$^2$, \\
\textbf{Fran\c{c}ois-Xavier Vialard}$^3$, \qquad \textbf{Gabriel Peyr{\'e}}$^2$ \\[2mm]
  1: Laboratoire de Mathématiques d’Orsay, CNRS, Universit\'e Paris-Saclay, Orsay, France \\
  2: ENS, PSL University, Paris, France\\
  3: Univ. Gustave Eiffel, CNRS, ESIEE Paris, Marne-la-Vallée, France
  }

\begin{document}

\maketitle

\input{abstract}
\input{intro}
\input{entropic-asymptotics}

\input{rates-sinkhorn-div}
\input{richardson}
\input{numerics}
\input{conclusion}


\section*{Acknowledgments}
The works of Pierre Roussillon, Flavien L\'eger and Gabriel Peyr\'e is supported by the ERC grant NORIA and by the French government under management of Agence Nationale de la Recherche as part of the “Investissements d’avenir” program, reference ANR19-P3IA-0001 (PRAIRIE 3IA Institute).

\bibliographystyle{plain}
\bibliography{paper_2020.bib}

\newpage
\appendix
\input{supplementary}

\end{document}

%% file: abstract.tex
\begin{abstract}
The squared Wasserstein distance is a natural quantity to compare probability distributions in a non-parametric setting. This quantity is usually estimated with the plug-in estimator, defined via a discrete optimal transport problem which can be solved to $\epsilon$-accuracy by adding an entropic regularization of order $\epsilon$ and using for instance Sinkhorn's algorithm.
In this work, we propose instead to estimate it with the Sinkhorn divergence, which is also built on entropic regularization but includes debiasing terms. 
We show that, for smooth densities, this estimator has a comparable sample complexity but allows higher regularization levels, of order $\epsilon^{1/2}$, which leads to improved computational complexity bounds and a strong speedup in practice. Our theoretical analysis covers the case of both randomly sampled densities and deterministic discretizations on uniform grids. We also propose and analyze an estimator based on Richardson extrapolation of the Sinkhorn divergence which enjoys improved statistical and computational efficiency guarantees, under a condition on the regularity of the approximation error, which is in particular satisfied for Gaussian densities.
We finally demonstrate the efficiency of the proposed estimators with numerical experiments.
\end{abstract}

%% file: intro.tex
\section{Introduction}\label{sec:introduction}

Certain tasks in machine learning (implicit generative modeling~\cite{mohamed2016learning}, two-sample testing~\cite{ramdas2017wasserstein}, structured prediction~\cite{frogner2015learning}) and imaging sciences (shape matching~\cite{glaunes2004diffeomorphic}, computer graphics~\cite{bonneel2011displacement}) require to quantify how much two probability densities $\mu,\nu \in \Pp(\RR^d)$ differ. The squared Wasserstein distance $W_2^2(\mu,\nu)$ (defined below) is often well suited for this purpose because of its appealing geometrical properties~\cite{villani2008optimal, santambrogio2015optimal, peyre2019computational} but it also raises important statistical and computational challenges. 
Indeed, in many practical settings, $\mu$ and $\nu$ are only accessed via empirical or discretized measures $\hat \mu_n,\hat \nu_n$ composed of $n$ atoms. A standard workaround is to use the \emph{plug-in estimator} $W_2^2(\hat \mu_n,\hat \nu_n)$, but although it is efficient when $\mu$ and $\nu$ are discrete~\cite{sommerfeld2018inference, tameling2019empirical}, this estimator suffers from the curse of dimensionality when $\mu$ and $\nu$ have densities~\cite[Cor.~2]{weed2019sharp}, with an estimation error that scales as $n^{-2/d}$ as we show in Section~\ref{sec:statisticalestimation}.
Moreover, solving the discrete optimal transport problem is computationally demanding when $n$ is large, with a time complexity bound scaling as $n^2\log(n)/\epsilon^2$ to reach $\epsilon$-accuracy with Sinkhorn's algorithm~\cite{dvurechensky2018computational,altschuler2017near}. These drawbacks give a strong motivation to define and study alternative estimators for $W^2_2(\mu,\nu)$ when $\mu$ and $\nu$ admit smooth densities.

\paragraph{Entropic regularization of optimal transport.}
In this paper, we consider instead estimators based on the idea of entropic regularization of optimal transport~\cite{wilson1969use, erlander1990gravity,kosowsky1994invisible,cuturi2013sinkhorn}. When $\mu$ and $\nu$ have finite second moments, the entropy regularized optimal transport cost is defined as~
\begin{equation}
\label{eq:entropycost}
    T_\lambda(\mu,\nu) \eqdef \min_{\gamma \in \Pi(\mu,\nu)} \int_{(\RR^d)^2}  \Vert y-x\Vert_2^2\,\dd\gamma(x,y) + 2\lambda H(\gamma, \mu\otimes \nu)
\end{equation}
where $\Pi(\mu,\nu)$
is the set of transport plans between $\mu$ and $\nu$, $\lambda \geq 0$ is the regularization parameter, and $H(\gamma,\mu\otimes \nu)$ is the entropy of $\gamma$ with respect to the product measure $\mu \otimes \nu$ (see details in the Notations paragraph). The squared Wasserstein distance is defined as $W_2^2(\mu,\nu) \eqdef T_0(\mu,\nu)$. Entropic regularization has been popularized as a method to compute $W_2^2(\hat\mu_n,\hat\nu_n)$ efficiently or simply as a different notion of discrepancy between measures. In contrast, we use it as a tool to directly estimate $W_2^2(\mu,\nu)$. For this purpose, the choice $T_\lambda(\hat \mu_n,\hat \nu_n)$ is not ideal because its large bias requires to set $\lambda$ to a small value, leading to computational difficulties.

\paragraph{The proposed estimators.} The first estimator that we consider is $\hat S_{\lambda,n} = S_\lambda(\hat \mu_n,\hat \nu_n)$ where $S_\lambda$ is the \emph{Sinkhorn divergence}~\cite{ramdas2017wasserstein} defined as
\begin{equation}\label{eq:sinkhorndivergence}
S_\lambda(\mu,\nu) \eqdef T_\lambda(\mu,\nu) - \frac12\big( T_\lambda(\mu,\mu) + T_\lambda(\nu,\nu)\big).
\end{equation}
In previous work~\cite{feydy2019interpolating}, the debiasing terms have been theoretically justified as a mean to have $S_\lambda(\mu,\nu)\geq 0$ with equality when $\mu=\nu$, a property not satisfied by $T_\lambda$. In the present work, we show that they in fact allow, under regularity assumptions, to approximate $W_2^2(\mu,\nu)$ with an error of order $\lambda^2$, instead of $\lambda \log(1/\lambda)$ for the uncorrected quantity $T_\lambda$. We also consider the estimator $\hat R_{\lambda,n} = R_\lambda(\hat \mu_n,\hat \nu_n)$ where $R_\lambda$ is built from $S_\lambda$ via Richardson extrapolation as
\begin{equation}\label{eq:richardson}
R_\lambda(\mu,\nu) \eqdef 2S_\lambda(\mu,\nu) - S_{\sqrt{2}\lambda}(\mu,\nu).
\end{equation}
This estimator has a smaller approximation error in $o(\lambda^2)$ and potentially in $O(\lambda^4)$ under restrictive regularity assumptions.

\paragraph{Contributions.}

We make the following contributions:
\begin{itemize}[label={--}]
\item
In Section~\ref{sec:approximation}, we exploit the dynamical formulation of~\eqref{eq:entropycost} to show that $\vert S_\lambda(\mu,\nu) - W_2^2(\mu,\nu)\vert\leq \lambda^2 I$ where $I$ depends on the Fisher information of $\mu$, of $\nu$ and of the $W_2$-geodesic connecting them. We also give a second-order expansion of this approximation error and detail several situations where $I$ admits a priori bounds.
\item 
In Section~\ref{sec:plugin}, we prove a sample complexity bound for the plug-in estimator $W_2^2(\hat \mu_n,\hat \nu_n)$ of order $n^{-2/d}$ which has a tight exponent in contrast to the previously known rate $n^{-1/d}$. This is the baseline rate against which we compare the performance of $\hat S_{\lambda,n}$ and of $\hat R_{\lambda,n}$.
\item 
In Section~\ref{sec:Srandom}, we study the performance of the Sinkhorn divergence estimator $\hat S_{\lambda,n}$ given independent samples. We show that when $\lambda$ is properly chosen, it enjoys comparable sample complexity bounds and improved computational guarantees in a certain sense. We also study the performance when the marginals are discretized on a uniform grid in Section~\ref{sec:discretized}.
\item In Section~\ref{sec:extrapolation}, we study estimators based on Richardson extrapolation such as $\hat R_{\lambda,n}$. Under an abstract and stronger regularity assumption, this estimator enjoys better computational and sample complexity bounds than the plug-in estimator. We discuss this assumption and show that it is satisfied for Gaussian densities.
\item 
In Section~\ref{sec:numerics}, we perform numerical experiments that confirm the benefits of the proposed estimators and suggest that our theoretical results could be extended in several ways.
\end{itemize}

\paragraph{Previous Works.}
Without additional assumptions, no estimator achieves better statistical rates than the plug-in estimator~\cite[Thm. 3]{niles2019estimation}. Recent breakthroughs in statistical optimal transport~\cite{weed2019estimation, hutter2019minimax} have shown that other estimators can exploit smoothness assumptions to attain faster and nearly minimax estimation rates for $W_2$ or the dual potentials, but they are \emph{a priori} not computationally efficient. In contrast, our goal in this paper is to improve the computational efficiency of estimating $W_2^2(\mu,\nu)$ and we are not aiming at statistical optimality.

The idea of entropic regularization has a long history in computational optimal transport. It has been shown in~\cite{altschuler2017near, dvurechensky2018computational} that solving $T_\lambda(\hat \mu_n,\hat \nu_n)$ to $\epsilon$-accuracy requires $O(n^2 / (\lambda \epsilon))$ arithmetic operations using Sinkhorn's algorithm if the domain is bounded (see Appendix~\ref{app:computational}). We use this bound in our discussions on computational complexity because it cleanly quantifies how harder the problem becomes as $\lambda$ becomes smaller and also because Sinkhorn's algorithm is simple to implement and widely used in practice. Choosing $\lambda \asymp \epsilon/\log(n)$ allows in turn to estimate $W_2^2(\hat \mu_n,\hat \nu_n)$ to $\epsilon$-accuracy in $O(n^2\log(n)/\epsilon^2)$ operations~\cite{dvurechensky2018computational}. There are however various algorithms with better guarantees both for the regularized~\cite{dvurechensky2018computational,allen2017much, cohen2017matrix} and the unregularized problem~\cite{lahn2019graph, quanrud2019approximating, blanchet2018towards}. In our numerical experiments, we use Sinkhorn's iterations combined with Anderson's acceleration~\cite{anderson1965iterative,scieur2016regularized}, which in practice strongly speeds up convergence.

In front of the difficulty to estimate $W_2^2(\mu,\nu)$, researchers have also turned their attention to similar but more tractable discrepancy measures such as the sliced Wasserstein distance~\cite{rabin2011wasserstein} or the Sinkhorn divergence~\cite{ramdas2017wasserstein}, which can be both estimated at the parametric rate~\cite{genevay2019sample, mena2019statistical, manole2019minimax, nadjahi2019asymptotic}. However, there is ``no free lunch'' and unconditional statistical efficiency comes at the price of lack of adaptivity and discriminative power. In particular, it is known that when $\lambda\to \infty$, $S_\lambda(\mu,\nu)$ converges to the squared distance between the expectations of $\mu$ and $\nu$, which is a degenerate form of Kernel Mean Discrepancy~\cite{genevay2018learning, feydy2019interpolating}. This shows that the discriminative power of $S_\lambda$ decreases as $\lambda$ increases, but this phenomenon is not yet well understood nor quantified. From a theoretical viewpoint, we thus believe that seeing $S_\lambda$ as an estimator for $W_2^2$ allows to clarify the trade-offs at play in the choice of $\lambda$ between the statistical, approximation and computational errors. 

\paragraph{Notations.} 
For two probability measures $\mu,\nu \in \Pp(\RR^d)$, we denote by 
$\Pi(\mu,\nu)$ the set of transport plans between $\mu$ and $\nu$, which is the set of measures $\gamma \in \Pp(\RR^d \times \RR^d)$ with marginal $\mu$ (resp.~$\nu$) on the first (resp.~second) factor of $\RR^d \times \RR^d$. The quantity $H(\mu,\nu)$ is the entropy of $\mu$ relative to $\nu$, defined as $H(\mu, \nu) \eqdef \int \log(\dd \mu/\dd\nu)\dd\mu$ when $\mu$ is absolutely continuous with respect to $\nu$, and $+\infty$ otherwise.
 When $\mu$ has a density with respect to the Lebesgue measure, written $\mu(x)$, we define $H(\mu) \eqdef \int \log(\mu(x))\mu(x)\dd x$ its entropy relative to the Lebesgue measure. Finally, $\mu\otimes \nu \in \Pp(\RR^d\times \RR^d)$ is the product measure characterized by $(\mu\otimes \nu)(A\times B) = \mu(A)\nu(B)$ for any pair of Borel sets $A,B \subset \RR^d$. 

%% file: entropic-asymptotics.tex
\section{Refined approximation bound for the Sinkhorn divergence}\label{sec:approximation}

In this section, we study the approximation error of $S_\lambda$. To this goal, we leverage the dynamical formulation of entropic optimal transport~\cite{chen2016relation, gentil2017analogy,gigli2018benamoubrenier,conforti2019formula} which states that, for $\mu,\nu \in \Pp(\RR^d)$ absolutely continuous probability measures with compact support,
\begin{multline}\label{eq:EntropicDynamic}
    T_\lambda(\mu,\nu)  + {d\lambda}\log(2\pi \lambda) +\lambda (H(\mu)  + H(\nu)) =\\
    \inf_{\rho,v} \int_0^1 \int_{\RR^d}  \Big(\| v(t,x) \|^2_2 + \frac {\lambda^2} 4 \| \nabla_x \log(\rho(t,x)) \|^2_2\Big)  \rho(t,x) \, \dd x\, \dd t\,,
\end{multline}
where the infimum is taken over time-dependent probability measures $\rho(t,x)$ that interpolate between $\mu$ at $t=0$ and $\nu$ at $t=1$, and time-dependent vector fields $v(t,x)$ under the continuity equation constraint $\partial_t \rho(t,x) + \on{div}(\rho(t,x) v(t,x)) =  0$ where $\on{div}$ is the usual divergence operator. The first term in the r.h.s. of Eq.~\eqref{eq:EntropicDynamic} is the kinetic energy and the second is the Fisher information integrated in time.
For $\lambda\geq 0$, there exists a unique minimizer of the r.h.s.~\cite{gigli2018benamoubrenier} denoted by $\rho_\lambda$ and we define
\begin{equation}\label{eq:fisherinformation}
I_\lambda(\mu,\nu) \eqdef \int_0^1\int_{\RR^d} \Vert \nabla_x \log(\rho_\lambda(t,x))\Vert_2^2 \, \rho_\lambda(t,x)\,  \dd x \, \dd t \,.
\end{equation}
Remark that $I_0(\mu,\mu)$ is the Fisher information of $\mu$ and $I_0(\mu,\nu)$ is the Fisher information of the Wasserstein geodesic between $\mu$ and $\nu$.  Building on~\cite{conforti2019formula}, we next show that the Sinkhorn divergence approximates $W_2^2(\mu,\nu)$ with an error in $O(\lambda^2)$, as suggested by Eq.~\eqref{eq:EntropicDynamic}.
\begin{theorem}\label{th:bias} Assume that $\mu,\nu \in \Pp(\RR^d)$ have bounded densities and supports. It holds
$$
\big\vert S_\lambda(\mu,\nu) - W_2^2(\mu,\nu)\big\vert \leq \frac{\lambda^2}{4} \max\big\{ I_0(\mu,\nu), (I_0(\mu,\mu)+I_0(\nu,\nu))/2\big\}\,.
$$
If moreover $I_0(\mu,\nu),I_0(\mu,\mu),I_0(\nu,\nu)<\infty$ then $$ S_\lambda(\mu,\nu) - W_2^2(\mu,\nu) = \frac{\lambda^2}{4}\big(I_0(\mu,\nu)-(I_0(\mu,\mu)+I_0(\nu,\nu))/2\big) + o(\lambda^2).$$
\end{theorem}

\begin{proof}
Denote the right-hand side of \eqref{eq:EntropicDynamic} by $J_{\lambda^2}(\mu,\nu)$ and note that
$
    S_\lambda(\mu,\nu) - W_2^2(\mu,\nu) = (J_{\lambda^2}(\mu,\nu) - J_0(\mu,\nu))  - ( J_{\lambda^2}(\mu,\mu) + J_{\lambda^2}(\nu,\nu))/2
$ and
$J_0(\mu,\mu)=J_0(\nu,\nu)=0$. Since $\rho_0$ is feasible in Eq.~\eqref{eq:EntropicDynamic}, we have $J_0(\mu,\nu)\leq J_{\lambda^2}(\mu,\nu)\leq J_0(\mu,\nu)+(\lambda^2/4)I_0(\mu,\nu)$, hence the bound. For the second claim, we prove in Appendix~\ref{app:bias} (Lemma~\ref{lem:rightderivative}) that the right derivative at $0$ of $\sigma\mapsto J_\sigma$ is $\frac14 I_0(\mu,\nu)$, which justifies the Taylor expansion.
\end{proof}
The Fisher information of $\mu$ or $\nu$ can be bounded by assuming regularity of the densities, but bounding $I_0(\mu,\nu)$ is more subtle. Next, we bound $I_0(\mu,\nu)$ assuming regularity on the \emph{Brenier potential} $\varphi$, which is the convex function such that $\nabla \varphi$ is the optimal transport map from $\mu$ to $\nu$~\cite{santambrogio2015optimal}.

\begin{proposition}\label{cor:FirstFisherBound}
Let $\mu,\nu \in \Pp(\RR^d)$ be absolutely continuous with compact support. Assume that the Brenier potential $\varphi$ has a Hessian satisfying $0\prec \kappa \mathrm{Id} \preceq \nabla^2 \varphi \preceq K\mathrm{Id}$ and that $\nabla^2 \varphi$ is $L$-Lipschitz continuous, then
$
I_0(\mu,\nu) \leq 2 \kappa^{-1}(I_0(\mu,\mu) +  \kappa^{-2} L^2/3)\,.
$
In particular, if $\varphi$ is quadratic then $I_0(\mu,\nu) \leq  2\kappa^{-1}I_0(\mu,\mu)$. If $d = 1$, then
$
   I_0(\mu,\nu) \leq \frac 23  ( \kappa^{-1} I_0(\mu,\mu) +  K I_0(\nu,\nu))\,.
$
\end{proposition}
Sufficient conditions on the densities of $\mu$ and $\nu$ to guarantee bounds on $\nabla^2\varphi$ are known (e.g.\ bounds on their first derivative and on their log-densities over their convex support~\cite[Thm 3.3]{philippis2013mongeampre}).
However, the assumption that $\nabla^2\varphi$ is Lipschitz continuous is more demanding and potentially not sharp as it can be avoided when $d=1$. Note that the Brenier potential $\varphi$ is quadratic whenever the densities are in the same family of elliptically contoured distributions~\cite{bhatia2019bures}. For Gaussian densities, we show in Appendix~\ref{app:bias} that $I_0(\mu,\nu)$ admits an explicit expression, given in Section~\ref{sec:extrapolation}.

%% file: rates-sinkhorn-div.tex
\section{Performance analysis of the Sinkhorn divergence estimator}\label{sec:statisticalestimation}
In this section, we discuss the performance of the Sinkhorn divergence estimator in two situations: when we observe independent samples or when we have access to discretized densities. But first, we study the plug-in estimator, which is the baseline against which our estimators are compared.  
\subsection{Analysis of the plug-in estimator}\label{sec:plugin}

\paragraph{A tighter statistical bound for the plug-in estimator.}
Let us first study the rate of convergence of $W_2^2(\hat \mu_n,\hat\nu_n)$ towards $W_2^2(\mu,\nu)$ where $\hat \mu_n$ and $\hat \nu_n$ are empirical distributions of $n$ independent samples.
This is well-studied in the case $\mu=\nu$, but the case $\mu\neq\nu$ was not specifically covered in the literature except for discrete measures~\cite{sommerfeld2018inference}. 
\begin{theorem}\label{thm:plug-in}
If $\mu,\nu \in \Pp(\RR^d)$ are supported on a set of diameter $1$ then it holds
\begin{equation*}
    \Esp \big[ \vert W_2^2(\hat \mu_n, \hat\nu_n) -  W_2^2(\mu, \nu)\vert\big] \lesssim \begin{cases}
    n^{-2/d} &\text{if $d>4$,}\\
    n^{-1/2}\log(n) &\text{if $d=4$,}\\
    n^{-1/2} &\text{if $d<4$,}
\end{cases}
\end{equation*}
where the notation $\lesssim$ hides constants that only depend on the dimension $d$. Also, this estimator concentrates well around its expectation, in the sense that for all $t\geq 0$,
$$
\mathbf{P}\Big[ \vert W_2^2(\hat \mu_n,\hat \nu_n) -\Esp[ W_2^2(\hat \mu_n,\hat \nu_n)]\vert \geq t \Big]\leq 2\exp(-nt^2).
$$

\end{theorem}
To prove this result in Appendix~\ref{app:plugin}, we first upper bound the expected error by the Rademacher complexity of a certain set of convex and Lipschitz functions. We use Dudley's chaining and a bound on the covering number of this set of functions due to Bronshtein~\cite{bronshtein1976varepsilon} to conclude. The concentration bound is already present in a similar form in~\cite[Prop.~20]{weed2019sharp}.
 When $\mu=\nu$, this bound is well-known and has a sharp exponent~\cite{weed2019sharp, singh2018minimax, boissard2014mean, dudley1969speed, fournier2015rate}. However, perhaps surprisingly, this result implies that the plug-in estimator $W_2(\hat \mu_n,\hat \nu_n)$ (without the square) converges at the rate $n^{-2/d}$ when $\mu\neq \nu$, while only a bound in $n^{-1/d}$ (the rate when $\mu=\nu$) was known. This is the content of the following corollary. See Figure~\ref{fig:sample_complexity plug in} for a numerical illustration of these rates.
 \begin{corollary}\label{cor:SCW2}
 Assume that $\mu,\nu$ are supported on a set of diameter $1$ and satisfy $W_2(\mu,\nu)\geq \alpha>0$. Then $ \Esp \big[ \vert W_2(\mu_n,\nu_n) -W_2(\mu,\nu)\vert \big]$ enjoys the bound given in Theorem~\ref{thm:plug-in} multiplied by $1/\alpha$.
 \end{corollary}
 \begin{proof}
It is sufficient to take expectations in the following inequality :
$$
\vert W_2(\hat \mu_n,\hat \nu_n)-W_2(\mu,\nu)\vert = \frac{\vert W_2^2(\hat \mu_n,\hat \nu_n) -W_2^2(\mu,\nu)\vert }{W_2(\mu,\nu)+W_2(\hat \mu_n,\hat \nu_n)} \leq \frac{1}{\alpha} \vert W_2^2(\hat \mu_n,\hat \nu_n) -W_2^2(\mu,\nu)\vert.\qedhere
$$
\end{proof}

          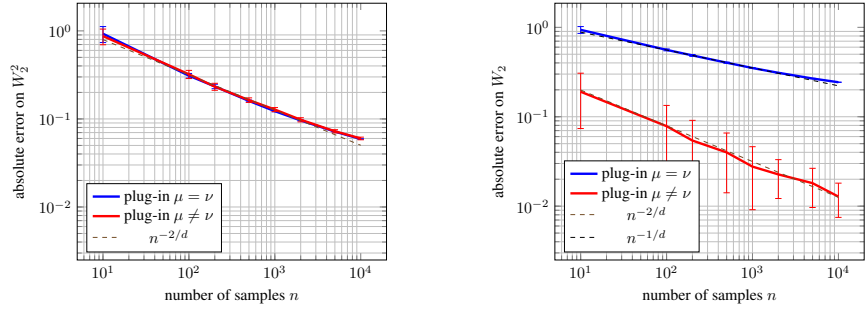
\begin{figure}
            \centering
                        \begin{subfigure}[b]{0.45\linewidth}
                        \centering
               \begin{tikzpicture}[scale = 0.6]
  \begin{loglogaxis}[xlabel = number of samples $n$, ylabel = absolute error on $W_2^2$, no markers, legend pos=south west, grid = minor, ymin = 1/400]
    \addplot[line width = 1.5pt, color = blue,error bars/.cd, y dir=both, y explicit] table [x= Nsamples, y= error_on_cost,y error = std_on_cost, col sep=space] {figures/plug_in_same_distr_elliptic_W22_dimension_5.csv};
    \addlegendentry{plug-in $\mu = \nu$}
        \addplot[line width = 1.5pt, color = red, error bars/.cd, y dir=both, y explicit] table [x= Nsamples, y= error_on_cost,y error = std_on_cost, col sep=space] {figures/plug_in_diff_distr_elliptic_W22_dimension_5.csv};
        \addlegendentry{plug-in $\mu \neq \nu$}
                    \addplot+[samples at = {10, 100, 200, 500, 1000, 2000, 5000, 10000}, dashed ] {2*(x)^(-2/5)};
                    \addlegendentry{$n^{-2/d}$}
\end{loglogaxis}
\end{tikzpicture}
            \end{subfigure}
            \begin{subfigure}[b]{0.45\linewidth}
            \centering
              \begin{tikzpicture}[scale = 0.6]
  \begin{loglogaxis}[xlabel = number of samples $n$, ylabel = absolute error on $W_2$, no markers, legend pos=south west, grid = minor, ymin = 1/400]
    \addplot[line width = 1.5pt, color = blue,error bars/.cd, y dir=both, y explicit] table [x= Nsamples, y= error_on_cost,y error = std_on_cost, col sep=space] {figures/plug_in_same_distr_elliptic_dimension_5.csv};
    \addlegendentry{plug-in $\mu = \nu$}
        \addplot[line width = 1.5pt, color = red, error bars/.cd, y dir=both, y explicit] table [x= Nsamples, y= error_on_cost,y error = std_on_cost, col sep=space] {figures/plug_in_diff_distr_elliptic_dimension_5.csv};
        \addlegendentry{plug-in $\mu \neq \nu$}
                    \addplot+[samples at = {10, 100, 200, 500, 1000, 2000, 5000, 10000}, dashed ] {0.5*(x)^(-2/5)};
                    \addlegendentry{$n^{-2/d}$}
                                        \addplot+[samples at = {10, 100, 200, 500, 1000, 2000, 5000, 10000}, dashed ] {1.4*(x)^(-1/5)};
            \addlegendentry{$n^{-1/d}$}
\end{loglogaxis}
\end{tikzpicture}
\end{subfigure}
\caption{Estimation error of the plug-in estimator for $\mu,\nu$ compactly supported with $d=5$ (as detailed in Appendix~\ref{app:additionalnumerics}). Left: error on the cost $W^2_2$ has rate $n^{-2/d}$ (Theorem~\ref{thm:plug-in}). Right: error on $W_2$ has rate $n^{-1/d}$ if $\mu = \nu$ and $n^{-2/d}$ if $\mu \neq \nu$ (Corollary~\ref{cor:SCW2}) with $\mathbf{E}_\mu[x]=0$ and $\mathbf{E}_\nu[x]=(1,\dots,1)$.
}
\label{fig:sample_complexity plug in}
\end{figure}

\paragraph{Computational complexity via Sinkhorn's algorithm.} In previous work~\cite{altschuler2017near, dvurechensky2018computational}, solving $T_\lambda(\hat \mu_n,\hat \nu_n)$ with $\lambda>0$ has been studied as a computationally efficient way to compute $T_0(\hat \mu_n,\hat \nu_n)$ and related quantities. One standard algorithm to compute $T_\lambda$ is Sinkhorn's algorithm, which can be interpreted as alternate block maximization on the dual of Eq.~\eqref{eq:entropycost}, see Appendix~\ref{app:computational}. Given two discrete marginals $\hat \mu_n = \sum_{i=1}^n p_i\delta_{x_i}$ and $\hat \nu_n=\sum_{i=1}^n q_j \delta_{y_j}$, let us define the cost matrix with entries $c_{i,j} = \frac12 \Vert x_i -y_j\Vert_2^2$. The iterates $u^{(k)}, v^{(k)} \in \RR^n$, $k\geq 1$ of Sinkhorn's algorithm are defined as follows: let $v^{(0)}=0 \in \RR^n$ and let
\begin{align}\label{eq:sinkhorniteration}
u^{(k)}_i = -\lambda \log \Big(\sum_{j=1}^n e^{(v^{(k-1)}_j-c_{i,j})/\lambda}q_j\Big) 
&&\text{and}&&
v^{(k)}_j = -\lambda \log \Big(\sum_{i=1}^n e^{(u^{(k)}_i-c_{i,j})/\lambda}p_i\Big).
\end{align}
An estimate for $\hat T_{\lambda,n} \eqdef T_\lambda(\hat \mu_n,\hat \nu_n)$ is then given by $\hat T^{(k)}_{\lambda,n} = 2\sum_{i=1}^n u_i^{(k)}p_i + v_i^{(k)}q_i$. These iterations enjoy the following guarantee, proved in~\cite{dvurechensky2018computational} (see details in Appendix~\ref{app:computational}).

\begin{proposition}\label{prop:sinkhornscomputational}
It holds $\vert \hat T^{(k)}_{\lambda,n} - \hat T_{\lambda,n} \vert\leq 2 \Vert c\Vert^2_\infty/(\lambda k)$ where $\Vert c\Vert_\infty = \max_{i,j} \Vert x_i-y_j\Vert_2^2/2$.
\end{proposition}
In particular, taking into account the fact that each iteration requires $O(n^2)$ arithmetic operations, Sinkhorn's algorithm returns an $\epsilon$-accurate estimation of $\hat T_{\lambda,n}$ in time $O(n^2 \Vert c\Vert_\infty^2 /(\lambda \epsilon))$. Moreover, if $\alpha>0$ is such that $p_i, q_j \geq \alpha/n$, we have the approximation bound $\vert \hat T_{\lambda,n} - \hat T_{0,n}\vert \leq 4\lambda \log(n/\alpha)$ which follows by bounding the relative entropy of admissible transport plans~\cite{altschuler2017near}. By fixing $\lambda = \epsilon/4(\log(n/\alpha))$, we thus obtain an $\epsilon$-accurate estimation of $\hat T_{0,n}$ in $O(n^2 \log(n/\alpha) \Vert c\Vert_\infty^2/\epsilon^2)$ operations.
As a consequence, by combining Theorem~\ref{thm:plug-in} and Proposition~\ref{prop:sinkhornscomputational}, we can thus give the following computational complexity bound to estimate $W_2^2(\mu,\nu)$ given random samples that takes into account the number of samples and the regularization level required to reach a certain accuracy.
\begin{proposition}\label{prop:plugincomputational}
Assume that $\mu,\nu$ are supported on a set of diameter $1$. Using $\hat T^{(k)}_{\lambda,n}$, an $\epsilon$-accurate estimation of $W_2^2(\mu,\nu)$ is achieved with probability $1-\delta$ in $\tilde O(\epsilon^{-\max\{6,d+2\}})$ operations, where $\tilde O$ hides poly-log factors in $1/\epsilon$ and $1/\delta$.
\end{proposition}
\begin{proof}[Proof idea]
We write $W_2^2\eqdef W_2^2(\mu,\nu)$, $\hat W_2^2\eqdef W_2^2(\hat \mu_n,\hat \nu_n)$ and consider the error decomposition
$$
\vert \hat T^{(k)}_{\lambda,n} - W_2^2\vert \leq 
\vert \hat T^{(k)}_{\lambda,n} - \hat T_{\lambda,n}\vert+\vert \hat T_{\lambda,n} - \hat W_2^2\vert + \vert \hat W_2^2 - \Esp [\hat W_2^2]\vert + \Esp \vert \hat W_2^2 - W_2^2]\vert
$$
where each term has been bounded in the previous discussion, see details in Appendix~\ref{app:plugin}.
\end{proof}

\subsection{Performance of the Sinkhorn divergence estimator given random samples}\label{sec:Srandom}
\paragraph{Statistical performance.}
Let us now turn to our object of interest which is the Sinkhorn divergence estimator $\hat S_{\lambda,n} \eqdef S_{\lambda}(\hat \mu_n,\hat \nu_n)$, defined from $n$ independent samples from $\mu$ and $\nu$. We note that all the results in this section also apply to the estimator $ T_{\lambda}(\hat \mu_n,\hat\nu_n) - (T_\lambda(\hat \mu_{n/2},\hat \mu_{n/2}')+T_\lambda(\hat \nu_{n/2},\hat \nu'_{n/2}))/2$ where $\hat \mu_{n/2}$ (resp.~$\hat \mu'_{n/2}$) is the empirical distribution of the first (resp.~second) half samples from $\mu$ (assuming $n$ even for conciseness), which is a natural alternative definition. The following result gives the expected error of the estimator $\hat S_{\lambda,n}$.
\begin{proposition}\label{prop:samplesinkhorn}
Let $\mu,\nu$ be supported on a set of diameter $1$ and assume that $\vert S_\lambda(\mu,\nu)-W_2^2(\mu,\nu)\vert \leq \lambda^2I$ for some $I>0$ (see guarantees in Section~\ref{sec:approximation}). Then, with the choice $\lambda = n^{\frac{-1}{d'+4}}$, it holds
$$
\Esp \big[ \vert \hat S_{\lambda,n} -  W_2^2(\mu,\nu)\vert\big]\lesssim n^{\frac{-2}{d'+4}}.
$$
where $d'=2\lfloor d/2\rfloor$ and $\lesssim$ hides a constant depending only on $I$ and $d$. Also, this estimator concentrates well around its expectation: for all $t,\lambda\geq 0$,
$
\mathbf{P}\Big[ \vert \hat S_{\lambda,n} -\Esp[ \hat S_{\lambda,n}]\vert \geq t \Big]\leq 2\exp(-nt^2/4).
$
\end{proposition}
Observe that when $d$ is large, the exponent $-2/(d'+4)$ is equivalent to $-2/d$ which is the rate of the plug-in estimator as shown in Theorem~\ref{thm:plug-in}. However, except for $d=1$, this exponent is slightly worse and we believe that this is due to a weakness in our bound. In fact, in our numerical experiments we observe that $\hat S_{\lambda,n}$ is in fact more statistically efficient than the plug-in estimator (cf.~Figure~\ref{fig:sampling}).

\paragraph{Computational performance.} An ideal theoretical goal would be to exhibit a computational advantage for using $\hat S_{\lambda,n}$ in the sense of Proposition~\ref{prop:plugincomputational}, but unfortunately the statistical bound in Proposition~\ref{prop:samplesinkhorn} is not strong enough to allow for such a result. Still, there is a clear computational advantage in using $\hat S_{\lambda,n}$ which is that to attain an accuracy $\epsilon$, it requires a regularization level $\lambda$ of order $\epsilon^{1/2}$ instead of $\epsilon$ for the plug-in estimator. This advantage can be formalized as follows, where $\hat S^{(k)}_{\lambda,n}$ is the estimation of $\hat S_{\lambda,n}$ obtained after $k$ Sinkhorn's iterations.
\begin{proposition}\label{prop:timecomplexitySinkhorn}
Under the assumptions of Proposition~\ref{prop:samplesinkhorn}, an $\epsilon$-accurate estimation of $W_2^2(\mu,\nu)$ can be obtained with probability $1-\delta$ in $\tilde O(\epsilon^{-(d'+5.5)})$ computations via $\hat S^{(k)}_{\lambda,n}$ where $d'=2\lfloor d/2\rfloor$ and $\tilde O$ hides a poly-log factor in $1/\delta$. Given $n$ samples, both estimators can achieve with probability $1-\delta$ an accuracy $\epsilon \asymp n^{-2/(d'+4)}$, but in time $\tilde O(n^2\epsilon^{-1.5})$ via $\hat S^{(k)}_{\lambda,n}$ and in time $\tilde O(n^2\epsilon^{-2})$ via $T^{(k)}_{\lambda,n}$.
\end{proposition}
\begin{proof}[Proof idea]
For $\hat T^{(k)}_{\lambda,n}$, we consider the error decomposition of Proposition~\ref{prop:plugincomputational}, while for $\hat S^{(k)}_{\lambda,n}$, we write
$$
\vert \hat S^{(k)}_{\lambda,n} - W_2^2\vert \leq \vert  \hat S^{(k)}_{\lambda,n} -  \hat S_{\lambda,n}\vert +
\vert \hat S_{\lambda,n} - \Esp [\hat S_{\lambda,n}] \vert + 
\Esp \vert \hat S_{\lambda,n} - S_{\lambda}\vert +
\vert S_{\lambda} - W_2^2\vert.
$$
The key difference with the decomposition in the proof of Proposition~\ref{prop:plugincomputational} is that the error induced by the entropic regularization is bounded on the population quantities instead of the empirical ones. These terms have been bounded in the previous discussion, see details in Appendix~\ref{app:sinkhorndivergencesstat}.
\end{proof}

\subsection{Performance of the Sinkhorn divergence estimator given densities discretized on grids} \label{sec:discretized}
In this section, we consider the case where the marginals $\mu$ and $\nu$ are not randomly sampled, but instead are accessed via their discretized densities which is the common situation in imaging sciences. We show a stability property of the entropy regularized optimal transport which leads to improved error bounds compared to the plug-in estimator.

For simplicity, we consider measures on the $d$ dimensional torus $\TT^d=(\RR/\ZZ)^d$ with its usual distance denoted by $\Vert [x-y]\Vert_2$. For a measure $\mu \in \Pp(\TT^d)$ its discretization $\mu_h$ at resolution $h=1/m$ for an integer $m$ is the discrete measure with $n=m^d$ atoms supported on the regular grid $(\ZZ/m\ZZ)^d$ which gives to each point the mass of $\mu$ on its surrounding cell.
The following approximation result suggests that regularizing the optimal transport problem increases the stability under such a discretization.
\begin{proposition}[Stability under discretization]\label{prop:grid}
Assume that $\mu,\nu \in \Pp(\TT^d)$ admit $M$-Lipschitz continuous log-densities and let $C>0$ be any constant. If  $h(M+\lambda^{-1})\leq C$ then
\begin{align*}
\vert  T_\lambda(\mu_h,\nu_h)-T_\lambda(\mu,\nu)\vert \lesssim \min\{h, h^2(\lambda^{-1} +M+1)\}
\end{align*}
where $\lesssim$ hides constants that only depend on $d$ and $C$.
\end{proposition}

This bound implies an error of order $h^2$ for the entropy regularized problem while it is not known whether such a bound is possible for $\lambda=0$, where a naive analysis suggests a bound of order $h$. When combined with the approximation error and the analysis of Sinkhorn's iterations, this yields the following performance guarantees for $S_\lambda(\mu_h,\nu_h)$ as defined in Eq.~\eqref{eq:sinkhorndivergence}.

\begin{proposition}\label{prop:gridperf}
Assume that $\mu,\nu \in \Pp(\TT^d)$ admit Lipschitz continuous log-densities and that $I_0(\mu,\nu) $ is finite. We can estimate $W_2^2(\mu,\nu)$ to $\epsilon$-accuracy:
\begin{itemize}[label={--}]
\item with $T_\lambda(\mu_h,\nu_h)$ in time $\tilde O(\epsilon^{-(2d+2)})$ by setting $h \asymp \epsilon$ and $\lambda \asymp \epsilon/\log(1/\epsilon)$,
\item with $S_\lambda(\mu_h,\nu_h)$ in time $O(\epsilon^{-(3d/2+3/2)})$ by setting  $h \asymp \epsilon^{3/4}$ and $\lambda \asymp \epsilon^{1/2}$.
\end{itemize}
\end{proposition}
This result suggests that $S_\lambda(\mu_h,\nu_h)$ estimates $W_2^2(\mu,\nu)$ both faster and more accurately than $T_\lambda(\mu_h,\nu_h)$ for their respective optimal $\lambda$, and this behavior is observed in numerical experiments (cf.~Figure~\ref{fig-deterministic}). Our aim with Proposition~\ref{prop:gridperf} is to illustrate the potential usefulness of the debiasing terms beyond the random sampling setting, but we stress that we are just comparing simple upper bounds which are not intended to be the best possible (in particular, we are not exploiting the fact that the computational cost of each Sinkhorn iteration could be reduced from $O(n^2)$ to $O(n\log(n))$ using discrete convolutions~\cite[Sec.~6.3.1]{berman2017sinkhorn}). In fact, in a similar setting, a completely different analysis of Sinkhorn's iterations is carried in~\cite[Cor.1.4]{berman2017sinkhorn}, where a time complexity in $\tilde O(\epsilon^{-(2d+1)})$ is derived for $T_\lambda(\mu_h,\nu_h)$.

%% file: richardson.tex
\section{Towards faster estimation with Richardson extrapolation}\label{sec:extrapolation}
The systematic bias induced by the Fisher information terms in Theorem~\ref{th:bias} can be removed using Richardson extrapolation~\cite{joyce1971survey, richardson1911approximate}, which usefulness in machine learning was recently pointed out  in~\cite{bach2020effectiveness}.
This technique consists in taking linear combinations of  $S_\lambda$ for various values of $\lambda>0$ in order to estimate $S_0$, by cancelling the successive terms of the Taylor expansion of $S_\lambda$ at $0$. Since in our context the first term of $S_\lambda-S_0$ is of order $\lambda^2$, this suggests to define (among other possible choices)
$
R_\lambda \eqdef 2S_\lambda - S_{\sqrt{2}\lambda}.
$
Indeed, whenever $S_\lambda = S_0 + \lambda^2 I +o(\lambda^2)$ for some $I\in \RR$, such as under the assumptions of Theorem~\ref{th:bias}, this quantity satisfies $R_\lambda = S_0 +o(\lambda^2)$.

\paragraph{Efficiency of $R_\lambda$ under an abstract assumption.} A difficulty with $R_\lambda$, or other extrapolated estimators, is that understanding their performance requires a fine understanding of the regularization path $\lambda \mapsto S_\lambda$. By remarking that in Eq.~\eqref{eq:EntropicDynamic}, $\lambda$ appears only via its square after debiasing, we might conjecture that if $S_\lambda$ admits a $4$th order Taylor expansion at $\lambda=0$, then the third term vanish. Before giving some arguments in favor of this property, let us state what it implies in terms of the performance of $\hat R_{\lambda,n} = \hat S_{\lambda,n} - \hat S_{\sqrt{2}\lambda,n}$, the extrapolation of the estimator $\hat S_{\lambda,n}$.

\begin{proposition}\label{prop:richardsonsample}
Assume that $\mu,\nu$ are compactly supported, that
$S_\lambda(\mu,\nu) - W_2^2(\mu,\nu) =  \lambda^2 I  +O(\lambda^4)$ for some $I\in \RR$ and let $d'=2\lfloor d/2\rfloor$. Then with $\lambda \asymp n^{-1/(d'+8)}$ it holds
$$
\Esp\big[ \vert \hat R_{\lambda,n} - W_2^2(\mu,\nu) \vert \big] \lesssim n^{-4/(d'+8)}. 
$$
Moreover, with probability $1-\delta$, this estimator returns an $\epsilon$-accurate estimation of $W_2^2(\mu,\nu)$ with $\tilde O(\epsilon^{-(d'+11)/2})$ computations via Sinkhorn's algorithm where $\tilde O$ hides poly-log factors in $1/\delta$.
\end{proposition}
\begin{proof}
 We use Lemma~\ref{lem:samplecomplexityentropy} to get
 $$
\Esp[ \vert \hat R_{\lambda,n} - W_2^2(\mu,\nu)\vert] \leq  \Esp[\vert \hat R_{\lambda,n} - R_{\lambda}(\mu,\nu)\vert ] + \vert R_\lambda(\mu,\nu) - W_2^2(\mu,\nu)\vert \lesssim (1+\lambda^{-d'/2})n^{-1/2} + \lambda^4
$$
and optimize the bound in $\lambda$. For the last claim we proceed as in the proof of Proposition~\ref{prop:plugincomputational}.
\end{proof}
 Under this abstract assumption, there is thus a clear statistical improvement over the plug-in estimator for $d>8$ and a computational improvement for $d>6$. Notice that a similar performance analysis could be done in the deterministic setting of Section~\ref{sec:discretized}. In the rest of this section we discuss the assumption of Proposition~\ref{prop:richardsonsample}. First we show that it is satisfied in the Gaussian case and second we propose formal calculations towards a $4$th order Taylor expansion of $T_\lambda$.
 
 \paragraph{Gaussian case.}
 Let $\mu =\mathcal{N}(a,A)$ and $\nu = \mathcal{N}(b,B)$ be Gaussian probability distributions with means $a,b \in \RR^d$ and positive definite covariances $A,B \in \RR^{d\times d}$.
In this case, it is well known that $W_2^2(\mu,\nu) = \Vert a-b\Vert_2^2 + \bures^2(A,B)$ where $\bures^2(A,B) \eqdef \tr(A) +\tr(B) - 2\tr(S)$ with $S=(A^{1/2}BA^{1/2})^{1/2}$ is the squared Bures distance~\cite{bhatia2019bures}. More recently, an explicit expression for
$T_\lambda(\mu,\nu)$ was derived in~\cite{anonymous2020gaussian,chen2015optimal,mallasto2020entropy}. By a Taylor expansion of this expression (see Appendix~\ref{app:gaussian}), we find that
$$
S_\lambda(\mu,\nu) - W_2^2(\mu,\nu) = -\frac{\lambda^2}{8} \bures^2(A^{-1},B^{-1}) + \frac{\lambda^4}{384}\bures^2(A^{-3},B^{-3}) +O(\lambda^5).
$$
 This expansion shows that the hypotheses of Proposition~\ref{prop:richardsonsample} are satisfied (to the exception of the compactness assumption, but note that sample complexity bounds for $S_\lambda$ are also known in this case~\cite{mena2019statistical}). Also we can explicitly compute the Fisher information $I_0(\mu,\nu) = \tr (S^{-1})$ (Appendix~\ref{app:bias}) which shows that the second order term is consistent, as it must, with the expansion in Theorem~\ref{th:bias}.

\paragraph{Formal fourth order expansion.}
Denoting $J_{\lambda^2}(\mu,\nu)$ the r.h.s.~of Eq.~\eqref{eq:EntropicDynamic}, we show in Lemma~\ref{lem:rightderivative} that $\sigma \mapsto J_\sigma$ admits a right derivative at all $\sigma\geq 0$ which is the Fisher information $\frac14 I_{\sqrt{\sigma}}(\mu,\nu)$ defined in Eq.~\eqref{eq:fisherinformation}. Thus, if we assume that $\sigma \mapsto I_{\sqrt{\sigma}}(\mu,\nu)$ admits a right derivative $I'_0$ at $0$, then it holds
\begin{multline*}
T_\lambda(\mu,\nu)  = T_0(\mu,\nu) -{d\lambda}\log(2\pi \lambda)-\lambda (H(\mu)  + H(\nu)) + \frac{\lambda^2}{4} I_0(\mu,\nu) + \frac{\lambda^4}{8}  I'_0 +o(\lambda^4)\,,
\end{multline*}
where $I'_0=\frac{\dd}{\dd(\lambda^2)}I_\lambda(\mu,\nu)\vert_{\lambda=0} = \int_0^1 \int_{\RR^d} (\Vert \nabla \log \rho_0 \Vert^2 - 2\Delta \rho_0/\rho_0)\delta_{\lambda^2} \rho_\lambda\vert_{\lambda=0} \, dx$ is the variation of Fisher information in the direction of $\delta_{\lambda^2} \rho_\lambda\vert_{\lambda=0_+}$, the first variation of $\rho_\lambda$ w.r.t.~$\lambda^2$. 
Hence under this abstract regularity assumption on $I_{\sqrt{\sigma}}(\mu,\nu)$, the result of Proposition~\ref{prop:richardsonsample} holds true.

%% file: numerics.tex
\section{Numerical experiments}\label{sec:numerics}
In this section, we assess the statistical and computational efficiency of the proposed estimators on synthetic problems\footnote{The code to reproduce these experiments is available at this webpage~\url{https://gitlab.com/proussillon/wasserstein-estimation-sinkhorn-divergence}.}. While this is what our theory controls, the error on the scalar $W_2^2(\mu,\nu)$ is not a suitable quantity to plot as it might vanish spuriously as we vary other parameters (such as $n$ or~$\lambda$), which hinders interpretation of the plots (see Appendix~\ref{app:additionalnumerics}). Instead, we propose to observe a more stringent and stable quantity, namely the $L_1$ error on the estimated dual potential $\varphi$, which is the Lagrange multiplier associated to the first marginal constraint in Eq.~\eqref{eq:entropycost}. This dual potential is the gradient of $W_2^2(\mu,\nu)$ with respect to $\mu$~\cite[Prop.~7.17]{santambrogio2015optimal}, a quantity of high interest when training machine learning models with $W_2^2$ as a loss function.

Specifically, given $v^{(k)} \in \RR^n$ obtained after $k$ Sinkhorn's iterations with discrete marginals $\mu_n,\nu_n$ as in Eq.~\eqref{eq:sinkhorniteration}, we define the  function $\hat u_{\mu,\nu}(x) = -\lambda \log (\sum_{j=1}^n e^{(v^{(k)}_j-\frac12 \Vert x-y_j \Vert_2^2)/\lambda}q_j)$. The quantity we plot is $\int \vert \hat \varphi_{\lambda,n}(x) - \varphi(x)\vert \dd\mu(x)$ estimated via Monte Carlo integration or on a fine grid, where $\hat \varphi_{\lambda,n}$ is defined as follows: (i) $\hat \varphi_{\lambda,n} = 2\hat u_{\mu,\nu}$ for the biased estimator $\hat T_{\lambda,n}$, (ii) $\hat \varphi_{\lambda,n} =2\hat u_{\mu,\nu} - (\hat u_{\mu,\mu'}+\hat v_{\mu,\mu'})$ for the debiased estimator $\hat S_{\lambda,n}$ and (iii) $2\hat \varphi_{\lambda,n} - \hat \varphi_{\sqrt{2}\lambda,n}$ for the extrapolated estimator $\hat R_{\lambda,n}$.

\paragraph{Random sampling.} Figure~\ref{fig:sampling} shows the approximation error for the estimators $T_\lambda$, $S_\lambda$ and $R_\lambda$ in the random sampling setting. Here, $\mu,\nu\in \Pp(\RR^d)$ with $d=5$ are smooth elliptically contoured distributions with compact support and are such that the optimal potential $\varphi$ is quadratic and admits a closed-form, as well as the transport cost (see Appendix~\ref{app:additionalnumerics}). These properties guarantee that the conclusions of Proposition~\ref{prop:samplesinkhorn} apply. As expected, for a given $\lambda$, $S_\lambda$ and $R_\lambda$ have a much smaller bias than $T_\lambda$ (left plot). Looking at the performance as a function of $\lambda$ (middle plot), we see that the error is minimal for some $\lambda^*$ that is much larger than what is needed for $T_\lambda$ to achieve a comparable accuracy. Also, choosing the best $\lambda^*$ for each $n$ (right panel), we see that $S_{\lambda^*}$ has the same rate as the plug-in estimator (estimated with $T_\lambda$ with a small $\lambda$), with a better constant. We remark that $R_\lambda$ does not converge faster, which does not contradict ours results since we have no guarantee on the specific quantity plotted here. 

\begin{figure}
  \centering
  
  \begin{subfigure}[b]{0.31\textwidth}
    \begin{tikzpicture}[scale = 0.5]
\begin{loglogaxis}[xlabel = number of samples $n$, ylabel = $L^1$ error on the first potential, no markers, legend pos=south west, grid = minor]

     \addplot+[ line width = 1.5pt, color = blue, error bars/.cd, y dir=both, y explicit] table [x= Nsamples, y= error_on_potentials, col sep=space, y error = std_on_potentials] {figures/OT_lambda_1_dimension_5.csv};
  \addlegendentry{$T_\lambda$}
  \addplot+[line width = 1.5pt, color = red, error bars/.cd, y dir=both, y explicit] table [x= Nsamples, y= error_on_potentials, col sep=space, y error = std_on_potentials] {figures/S_lambda_1_dimension_5.csv};
  \addlegendentry{$S_\lambda$}
  \addplot+[line width = 1.5pt, color = brown, error bars/.cd, y dir=both, y explicit] table [x= Nsamples, y= error_on_potentials, col sep=space, y error = std_on_potentials] {figures/R1_lambda_1_dimension_5.csv};
  \addlegendentry{$R_\lambda$}
                \addplot[thick, black] table [x= Nsamples, y= error_on_potentials, col sep=space] {figures/OT_lambda_0.01_dimension_5.csv};
    \addlegendentry{plug-in}

\end{loglogaxis}
\end{tikzpicture}
\end{subfigure}
\hfill
\begin{subfigure}[b]{0.31\textwidth}
    \begin{tikzpicture}[scale = 0.5]
\begin{loglogaxis}[xlabel = $\lambda$, ylabel = $L^1$ error on the first potential, no markers, legend pos=north west, grid = both, ymin = 2/100]
  \addplot+[ line width = 1.5pt, color = blue, error bars/.cd, y dir=both, y explicit] table [x= blur, y= error_on_potentials,y error = std_on_potentials, col sep=space]  {figures/OT_Nsamples_1000_dimension_5.csv};
  \addlegendentry{$T_\lambda$}
  \addplot+[line width = 1.5pt, color = red, error bars/.cd, y dir=both, y explicit] table [x= blur, y= error_on_potentials,y error = std_on_potentials, col sep=space]{figures/SD_Nsamples_1000_dimension_5.csv};
  \addlegendentry{$S_\lambda$}
  \addplot+[line width = 1.5pt, color = brown, error bars/.cd, y dir=both, y explicit] table [x= blur, y= error_on_potentials,y error = std_on_potentials, col sep=space] {figures/R1_Nsamples_1000_dimension_5.csv};
  \addlegendentry{$R_\lambda$}
\end{loglogaxis}
\end{tikzpicture}
\end{subfigure}
\hfill
\begin{subfigure}[b]{0.31\textwidth}
     \begin{tikzpicture}[scale = 0.5]
  \begin{loglogaxis}[xlabel = number of samples $n$, ylabel = $L^1$ error on the first potential, no markers, legend pos=north east, grid = minor, ymin = 1/100]
    \addplot[line width = 1.5pt, color = blue] table [x= Nsamples, y= error_on_potentials, col sep=space] {figures/OT_lambda_opt_dimension_5.csv};
    \addlegendentry{$T_{\lambda^*}$}
        \addplot[line width = 1.5pt, color = red] table [x= Nsamples, y= error_on_potentials, col sep=space] {figures/S_lambda_opt_dimension_5.csv};
        \addlegendentry{$S_{\lambda^*}$}
            \addplot[line width = 1.5pt, color = brown] table [x= Nsamples, y= error_on_potentials, col sep=space] {figures/R1_lambda_opt_dimension_5.csv};
  \addlegendentry{$R_{\lambda^*}$}
                \addplot[thick, black] table [x= Nsamples, y= error_on_potentials, col sep=space] {figures/OT_lambda_0.01_dimension_5_xmax10000.csv};
                \addlegendentry{plug-in}

\end{loglogaxis}
\end{tikzpicture}
\end{subfigure}
\caption{$L^1$ estimation error on the first potential for $\mu,\nu$ smooth compactly supported distributions with $d=5$. Left: as function of $n$ for $\lambda=1$. Middle: as a function of $\lambda$, for $n=10000$. Right: as a function of $n$ for the optimal $\lambda^*(n)$. Error bars show the standard deviation on 30 realizations.}\label{fig:sampling}
\end{figure}
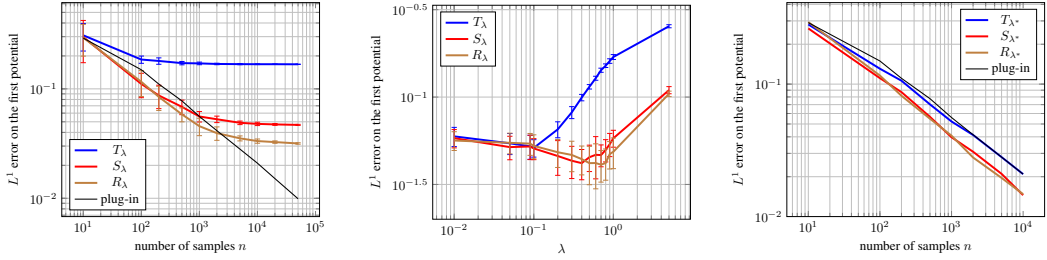

Overall, these estimators require less samples and a larger $\lambda$ to achieve a given accuracy compared to $T_\lambda$, which leads to substantial computational gains. 
This is illustrated on Figure~\ref{fig:computational_complexity} where for a target $L^1$ error on the potential, we chose the largest $\lambda$ and smallest $n$ that achieve this error, with $\lambda \in [0.1, 1]$ and $n \in [10, 100 000]$. We report the computational time using the Sinkhorn's iterations of Eq.~\eqref{eq:sinkhorniteration} stopped when the $\ell_1$-error on the marginals is below $10^{-5}$. We observe that for small target accuracies, the estimators $S_\lambda$ and $R_\lambda$ compare favorably to $T_\lambda$. In practical settings, one does not know \emph{a priori} the best choice for $\lambda$, but many machine learning tasks involving $W_2^2$ come with a performance criterion, in which case cross-validation can be used to select this parameter.

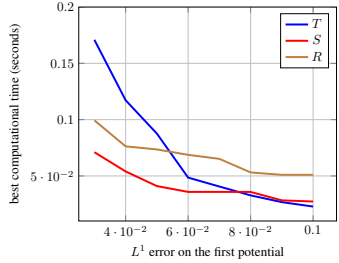
\begin{figure}[h]
              \centering
 \begin{tikzpicture}[scale = 0.5]
  \begin{axis}[xlabel = $L^1$ error on the first potential, ylabel = best computational time (seconds), no markers, legend pos=north east, grid = major, xmin = 0.025, ymin = 1/100, ymax = 0.2]
    \addplot[line width = 1.5pt, color = blue] table [x= Error , y= T, col sep=space] {figures/xp_ultime.csv};
    \addlegendentry{$T$}
    \addplot[line width = 1.5pt, color = red] table [x= Error , y= S, col sep=space] {figures/xp_ultime.csv};
    \addlegendentry{$S$}
   \addplot[line width = 1.5pt, color = brown] table [x= Error , y= R, col sep=space] {figures/xp_ultime.csv};
    \addlegendentry{$R$}
\end{axis}
\end{tikzpicture}
\caption{Best computational time achieved by the estimators to reach a given accuracy (after optimizing over $n$ and $\lambda$), for $\mu,\nu$ smooth compactly supported distributions with $d=5$. }
\label{fig:computational_complexity}
\end{figure}

\paragraph{Discretization on grids.}

Figure~\ref{fig-deterministic} shows the evolution of the errors for densities $(\mu,\nu)$ on the 1-D torus, the setting of Proposition~\ref{prop:gridperf}. In this case, one can compute efficiently the dual potentials $\varphi$ using cumulative functions~\cite{rabin2011transportation}.
This figure shows that, as expected, for a fixed $(h,\lambda)$ the error of $S_\lambda$ and $R_\lambda$ is systematically lower than that of $T_\lambda$. Even when selecting the optimal regularization $\lambda^\star(h)$ for each $h$ and for each method (which is a fair comparison), the error of $S_\lambda$ and $R_\lambda$ is still lower. 
Furthermore, the optimal parameter $\lambda^\star(h)$ is systematically larger for $S_\lambda$ and $R_\lambda$. 
Additional figures showing visual comparisons of the potentials and their approximations are provided in the appendix. 

\begin{figure}[h]
		\centering
		\includegraphics[width=\linewidth]{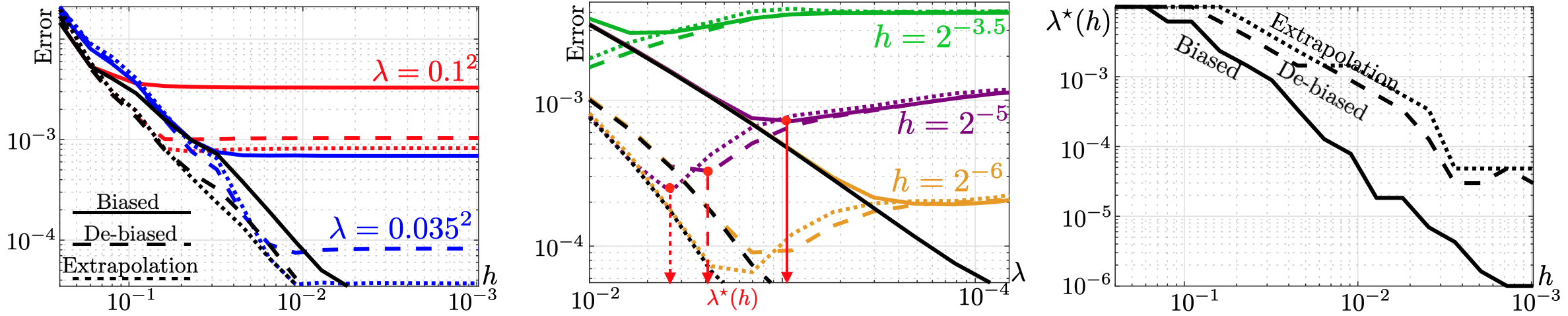}
	\caption{
	    Left: $L^1$ error on the first potential $\varphi$ as a function of the grid size $h$, for several value of $\lambda$. 
	    Middle: same error, displayed as a function of $\lambda$, for several grid sizes $h$.
	    Right: evolution of the optimal regularization parameter $\lambda^\star(h)$ as a function of the grid size $h$. 
	    }\label{fig-deterministic}
\end{figure}

%% file: conclusion.tex
\section{Conclusion and open questions}
In this paper we have exhibited the usefulness of entropic regularization with debiasing for the estimation of the squared Wasserstein distance: it may increase both accuracy and efficiency when the problem has a smooth nature. Numerical experiments suggest that the theory could be extended in several directions. First, the Sinkhorn divergence estimator appears at least as statistical efficient as the plug-in estimator, while our bound is slightly weaker. Also, the estimation of Kantorovich potentials seems to enjoy similar guaranties, but this is not covered by our theory.

%% file: supplementary.tex
\section*{Supplementary Material}
Supplementary material for the paper: “Faster Wasserstein Distance Estimation with the Sinkhorn Divergence” authored by Lénaïc Chizat, Pierre Roussillon, Flavien Léger, François-Xavier Vialard and Gabriel Peyré (NeurIPS 2020). This supplementary material is organized as follows:
\begin{itemize}
    \item Appendix~\ref{app:bias} contains the proofs of Section~\ref{sec:approximation},
    \item Appendix~\ref{app:computational} recaps the convergence analysis of~\cite{dvurechensky2018computational} to obtain Proposition~\ref{prop:sinkhornscomputational},
    \item Appendix~\ref{app:plugin} contains the proofs of Section~\ref{sec:plugin},
    \item Appendix~\ref{app:sinkhorndivergencesstat} contains the proofs of Section~\ref{sec:Srandom},
    \item Appendix~\ref{app:grid} contains the proofs of Section~\ref{sec:discretized},
    \item in Appendix~\ref{app:gaussian}, we derive the Taylor expansion for Gaussian distributions presented in Section~\ref{sec:extrapolation},
    \item finally, Appendix~\ref{app:additionalnumerics} contains details on the settings of the numerical experiments and additional figures.
\end{itemize}

\section{Bounds on the approximation error}\label{app:bias}

\paragraph{Dynamic entropy regularized optimal transport.}
Let us first justify how to obtain Eq.~\eqref{eq:EntropicDynamic} since our conventions are slightly different than in~\cite{conforti2019formula}. In that reference, for $\mu$ and $\nu$ absolutely continuous with compact support, the authors define
$$
\lambda C_\lambda(\mu,\nu) = \min_{\gamma \in \Pi(\mu,\nu)} \lambda H(\gamma, K)
$$
where $K = (2\pi \lambda)^{-d/2}\exp(-\Vert y-x\Vert_2^2/(2\lambda))\dd x\dd y$ is the heat kernel at time $\lambda/2$. In contrast, we can see from Eq.~\eqref{eq:entropycost} that
$$
\frac12 T_\lambda(\mu,\nu) = \min_{\gamma \in \Pi(\mu,\nu)} \lambda H(\gamma,\tilde K)
$$
where $\tilde K = \exp(-\Vert y-x\Vert_2^2/(2\lambda))\mu(x)\nu(y)\dd x \dd y$. We directly deduce that $\frac12 T_\lambda(\mu,\nu) = \lambda C_\lambda(\mu,\nu) - \lambda H(\mu)-\lambda H(\nu) -\frac{d\lambda}{2}\log(2\pi\lambda)$. Thus Eq.~\eqref{eq:EntropicDynamic} follows by the dynamic formulation of entropy regularized optimal transport in~\cite{conforti2019formula} which reads
\begin{align*}
\lambda C_\lambda(\mu,\nu) -\frac\lambda2 H(\mu) -\frac\lambda2H(\nu) 
&= \min_{\rho, v} \int_0^1 \int_{\RR^d} \big( \frac12\Vert v(t,x)\Vert_2^2 + \frac{\lambda^2}{8} \Vert \nabla_x \log \rho(t,x)\Vert_2^2 \big)\rho(t,x)\dd x \dd t
\end{align*}
where the constraints on $(\rho,v)$ are as in Eq.~\eqref{eq:EntropicDynamic}. Note that $\nabla_x \log \rho$ refers to the \emph{weak logarithmic gradient} of $\rho$, which in particular does not requires $\rho>0$ to be well defined, but only that for almost every $t\in [0,1]$, $\rho(t,\cdot)$ admits a distributional gradient which is an absolutely continuous measure with respect to $\rho(t,\cdot)$, and $\nabla_x \log \rho_t \eqdef \frac{\dd \nabla \rho_t}{\dd \rho_t}$ refers to its density with respect to $\rho_t$ (see e.g.~\cite{gianazza2009wasserstein}).
\paragraph{First order expansion.} Let us state and prove a lemma that intervenes in the proof of Theorem~\ref{th:bias}. Arguments towards this expansion appeared in~\cite[Theorem 1.6]{conforti2019formula} but under an abstract twice-differentiability assumption that is not needed in our statement.
\begin{lemma}\label{lem:rightderivative}
Assume that $\mu,\nu \in \Pp(\RR^d)$ have bounded densities and supports. It holds
$$
\frac{\dd}{\dd \sigma } J_{\sigma}\vert_{\sigma =0_+} = \frac14 I_0(\mu,\nu)
$$
where, as in the proof of Theorem~\ref{th:bias}, $J_{\lambda^2}(\mu,\nu)$ refers to the right-hand side of \eqref{eq:EntropicDynamic}. More generally, the right derivative of $\sigma \mapsto J_\sigma$ exists for all $\sigma \geq 0$ and equals $\frac14 I_{\sqrt{\sigma}}(\mu,\nu)$.
\end{lemma}
\begin{proof}
Since $\sigma \mapsto J_\sigma(\mu,\nu)$ is defined as an infimum of affine functions in $\sigma$, it is concave. Let $(\sigma_n)_{n\in \mathbb{N}}$ be a decreasing sequence of positive real numbers converging to $0$ and let 
$$
\alpha_n = \frac{J_{\sigma_n}-J_0}{\sigma_n}.
$$
By concavity, $\alpha_n$ is non-decreasing and admits a limit $J'_0 = \frac{\dd}{\dd \sigma} J_{\sigma}\vert_{\sigma=0_+}$ that is the right derivative of $J$ at $0$. Our goal is to show that $J'_0 =I_0(\mu,\nu)/4$.
By the argument in the proof of Theorem~\eqref{th:bias}, we have $\alpha_n\leq I_0(\mu,\nu)/4$ $\forall n$ thus $J'_0 \leq I_0(\mu,\nu)/4$, so we just have to prove the other inequality.

Let $(\rho_n,v_n)_{n\geq 0}$ be a sequence of minimizers for the r.h.s.~of Eq.~\eqref{eq:EntropicDynamic} (which is in fact unique although we do not use that fact here~\cite{gigli2018benamoubrenier}) with $\lambda^2=\sigma_n$ and let $V_n = \int_0^1\int_{\RR^d} \Vert v_n\Vert^2_2\dd\rho_n$ and $I_n = \int_0^1\int_{\RR^d} \Vert \nabla \log(\rho_n)\Vert_2^2\dd\rho_n$. Since $V_n$ is uniformly bounded and converges to $V_0=W_2^2(\mu,\nu)$, we have that $\rho_n$ converges weakly (in duality with continuous functions with compact support) to $\rho_0$, the unique constant speed Wasserstein geodesic between $\mu$ and $\nu$ (see, e.g.~\cite[Cor.~4.10]{dolbeault2009new} or by an application of~\cite[Proposition~2.2]{gianazza2009wasserstein} as below). Moreover, since $V_n\geq V_0$, it holds
$$
\alpha_n =\frac{V_n - V_0}{\sigma_n} + \frac14 I_n \geq \frac14 I_n
$$
and in particular we have the uniform bound $I_n\leq I_0$. It follows by~\cite[Proposition~2.2]{gianazza2009wasserstein} applied to the quantity $I_n=\int_0^1\int_{\RR^d}\Vert \frac{\dd(\nabla \rho_n)}{\dd \rho_n}\Vert_2^2\dd\rho_n(x,t)$  that $\nabla \rho_n$, seen as a vector valued measure on $[0,1]\times \RR^d$, admits a weak limit denoted $\omega$ which is absolutely continuous with respect to $\rho_0$ and that $\lim\inf I_n \geq \int_0^1\int_{\RR^d} \Vert \frac{\dd \omega}{\dd \rho_0}\Vert_2^2\dd\rho_0(t,x)$. Since for any compactly supported function $\varphi \in \mathcal{C}^1([0,1]\times \RR^d;\RR^d)$ it holds $\int \mathrm{div}_x(\varphi) \dd\rho_n\to \int \mathrm{div}_x(\varphi) \dd\rho_0$ and $\int \varphi \cdot \dd(\nabla \rho_n)\to \int \varphi \cdot \dd \omega$, we have that $\omega = \nabla_x \rho_0$ and thus the previous integral is precisely the Fisher information of $\rho_0$ integrated in time. It follows that $\lim\inf I_n \geq I_0$ hence $J'_0\geq \frac14 I_0$ which concludes the proof.
Inspecting the above argument, we see that in fact it applies directly to the case $\sigma>0$ (except that of course the trajectory recovered as $n\to \infty$ is $\rho_{\sqrt{\sigma}}$), hence our second claim.
\end{proof}

\paragraph{Bounds on the Fisher information of the geodesic.} Let us now prove the bounds on the Fisher information of the Wasserstein geodesic that appear in Proposition~\ref{cor:FirstFisherBound}. The main idea is to express $I_0(\mu,\nu)$ in terms of the initial and final densities and the Brenier potential.

\begin{proof}[Proof of Proposition~\ref{cor:FirstFisherBound}]
Let us express $I_0(\mu,\nu)$ in terms of the densities $\rho_0$ and $\rho_1$ (of $\mu$ and $\nu$ respectively) and the Brenier potential $\varphi$ which is the convex function such that $(\nabla \varphi)_\#\mu = \nu$, i.e.~$\nu$ is the pushforward of $\mu$ by the map $\nabla \varphi$. Let $(\rho_t)_{t\in [0,1]}$ be the density of the $W_2$-geodesic between $\mu$ and $\nu$. We start with the conservation of mass formula which holds under our regularity assumptions:
$$
\rho_0(x) = \det(\nabla^2\varphi_t(x))\rho_t(\nabla \varphi_t(x)))
$$
where $\varphi_t(x)\eqdef (1-t)\Vert x\Vert^2_2/2 + t\varphi(x)$ is such that $(\nabla \varphi_t)_\# \rho_0 =\rho_t$. By taking the logarithm we get
$$
\log \rho_0(x) = \log\rho_t(\nabla \varphi_t(x)) + \log \det(\nabla^2\varphi_t(x)).
$$
Let us now take the gradient of this expression. We denote by $d^3\varphi(x):\RR^d\to \RR^{d\times d}$ the weak differential of $x\mapsto \nabla^2\varphi(x)$ (which exists for almost every $x$ and is bounded since $\nabla^2\varphi$ is assumed Lipschitz) and by $[d^3\varphi(x)]^*:\RR^{d\times d}\to \RR^d$ its adjoint. Using the fact that the differential of $A\mapsto \log\det A$ at $A$ is the scalar product with $A^{-1}$ we get that for almost every $x\in \RR^d$,
\begin{equation}\label{eq:gradlog}
    \nabla \log \rho_0(x) = \nabla^2\varphi_t(x) \nabla \log \rho_t(\nabla \varphi_t(x)) + [d^3\varphi_t(x)]^* [\nabla^2 \varphi_t(x)]^{-1}.
\end{equation}
It follows that
\begin{align*}
    I_0(\mu,\nu) 
    &= \int_0^1\int_{\RR^d} \Vert \nabla \log \rho_t(x)\Vert_2^2 \rho_t(x) \dd x\dd t\\
    & =\int_0^1 \int_{\RR^d} \Vert \nabla \log \rho_t(\nabla \varphi_t(x)) \Vert_2^2 \rho_0(x)\dd x\dd t\\
    & = \int_0^1 \int_{\RR^d} \Vert [\nabla^2\varphi_t(x)]^{-1} \nabla \log \rho_0(x) - t  [\nabla^2\varphi_t(x)]^{-1} [d^3\varphi(x)]^* [\nabla^2 \varphi_t(x)]^{-1} \Vert_2^2 \rho_0(x)\dd x\dd t
\end{align*}
where we have used the fact that $d^3\varphi_t(x)=td^3\varphi(x)$.
\paragraph{General case.} In the general case, we simply use the bounds $\nabla^2 \varphi_t(x)\succeq ((1-t)+t\kappa) \mathrm{Id}$ and $\Vert d^3\varphi(x)\Vert \leq L$ almost everywhere in operator norm and the identity $\vert a+b\vert^2 \leq 2\vert a\vert^2 + 2\vert b\vert^2$ valid for any $a,b\in \RR$ to get
\begin{align*}
I_0(\mu,\nu) &\leq 2\Big(\int_0^1 \frac{\dd t}{(1+(\kappa-1))^2}\Big)I_0(\mu,\mu) + 2\Big(\int_0^1 \frac{t^2\dd t}{(1+(\kappa-1))^4}\Big)L^2 \\
& = 2\kappa^{-1}I_0(\mu,\mu) +(2/3)\kappa^{-3}L^2.
\end{align*}

\paragraph{One dimensional case.} When $d=1$, from Eq.~\eqref{eq:gradlog} at time $t=1$, we get
\[
\varphi'''(x) = \nabla \log \rho_0 (x) \varphi''(x) - \nabla \rho_1(\nabla\varphi(x)) \varphi''(x)^2.
  \]
  Plugging this expression in the previous integral leads to:
  \[
I_0(\mu,\nu) = \int_{\RR}\int_0^1\Bigg(\frac{(1-t)\nabla\log\rho_0 + t\nabla \log \rho_1(\nabla\varphi(x))\varphi''(x)^2}{\big((1-t) + t\varphi''(x) \big)^2}\Bigg)^2\dd t\rho_0(x)\dd x
\]
With the valid change of variables $1-s = \frac{t\varphi''(x)}{(1-t) + t\varphi''(x)}$ (and thus $s = \frac{1-t}{(1-t) + t\varphi''(x)}$), we obtain:

\[
  \begin{aligned}
    I_0(\mu,\nu) &= \int_\RR \int_0^{1}\frac{1}{\varphi''(x)}\Big((1-s)\nabla \log\rho_0(x) + s \nabla\log\rho_1(\varphi'(x))\varphi''(x)\Big)^2\dd s \rho_0(x) \dd x \\
    &= \int_\RR \int_0^{1}\frac{1}{\varphi''(x)}\Big((1-s)\nabla \log\rho_0(x) + s \nabla\log(\rho_1\circ\varphi')(x)\Big)^2\dd s \rho_0(x) \dd x
  \end{aligned}
\]
This leads to the bound $I_0(\mu,\nu) \leq (2/3)\kappa^{-1} I_0(\mu,\mu) + (2/3)KI_0(\nu,\nu)$ since $(\varphi')_\# \mu=\nu$.
\end{proof}

\paragraph{Gaussian case.} Let us now give the explicit expression of the Fisher information of the Wasserstein geodesic between Gaussian distributions, which is mentioned in Section~\ref{sec:extrapolation}. Whenever we deal with a positive semidefinite matrix $A$, the matrix $A^{1/2}$ refers to its unique positive semidefinite square root.
\begin{proposition}\label{prop:explicit_fisher}
If $\mu = {\cal N}(0,A), \: \nu = {\cal N}(0,B)$ then $I_0(\mu,\nu) = \tr S^{-1}$ with $S = (A^{1/2}BA^{1/2})^{1/2}$.
\end{proposition}
Remark in particular that the expansion in Theorem~\ref{th:bias} then gives 
$$
S_\lambda(\mu,\nu) - W_2^2(\mu,\nu) = \frac18(2I_0(\mu,\nu) - I_0(\mu,\mu) -I_0(\nu,\nu)) = \frac18(2\tr S^{-1} - \tr A^{-1} - \tr B^{-1})
$$
which is consistent, as it must, with the expansion in Section~\ref{sec:extrapolation}.

\begin{proof}
When the Brenier potential $\varphi = \frac12 x^\top Hx$ is quadratic, we have by the proof of Proposition~\ref{cor:FirstFisherBound}
$$
I_0(\mu,\nu) = \int_{\RR^d}\int_0^1 \Vert [\nabla^2\varphi_t(x)]^{-1} \nabla \log \rho_0(x)\Vert_2^2 \rho_0(x)\dd t\dd x.
$$
Putting ourselves in a basis diagonalizing $H$, the integration in time is explicit and we get
$$
I_0(\mu,\nu) = \int_{\RR^d} \Vert H^{-1/2} \nabla \log \rho_0(x)\Vert_2^2\rho_0(x)\dd x.
$$
It turns out that if $\mu = {\cal N}(0,A), \: \nu = {\cal N}(0,B)$, then $\varphi(x) = \frac{1}{2}x^THx $ where~\cite{bhatia2019bures} $$H = A^{-1/2}\Big(A^{1/2}BA^{1/2}\Big)^{1/2}A^{-1/2}$$ and thus
\[
  \begin{aligned}
    I_0(\mu,\nu) &= \int_{\RR^d} \Vert H^{-1/2}\nabla  \log \rho_0(x)\Vert_2^2\rho_0(x)\dd x.\\
    &=  \int_{\RR^d} \Vert H^{-1/2}A^{-1}x\Vert_2^2\rho_0(x)\dd x \\
    &= \Esp_{X \sim {\cal N}(0,A)}\Big[X^T\big(A^{-1}H^{-1}A^{-1}\big)X\Big] \\
    &= \tr\big( A^{-1}H^{-1}A^{-1}A\big)  =  \tr\big( A^{-1}H^{-1}\big) =  \tr\big( \big(A^{1/2}BA^{1/2}\big)^{-1/2}\big)
  \end{aligned}
\]
where the last row is obtained using \cite[Eq.~(378)]{petersen2012matrix}.
\end{proof}

\section{Computational complexity of Sinkhorn's algorithm}\label{app:computational}
In this appendix, we recall the computational complexity analysis of Sinkhorn's algorithm from~\cite{dvurechensky2018computational}, in order to state Proposition~\ref{prop:computationalcomplexity} exactly as per our needs (while this result is implicit in~\cite{dvurechensky2018computational}).  There is nothing specific in this analysis about the squared-distance cost so we just assume that the cost $c:\RR^d\times \RR^d\to \RR$ is continuous, keeping in mind that in our case, $c(x,y)=\frac12 \Vert y-x\Vert_2^2$. We also consider a compact set $\Xx \subset \RR^d$ and measures $\mu,\nu \in \Pp(\Xx)$ which are concentrated on this set. We consider the dual objective function of entropy regularized optimal transport~\cite{peyre2019computational}:
\begin{equation}\label{eq:dualentropyregularized}
F_\lambda(u,v) = \int_{\RR^d} u \dd\mu + \int_{\RR^d} v \dd\nu +\lambda \Big(1 - \int_{(\RR^d)^2} \exp((u(x)+v(y)-c(x,y))/\lambda)\dd\mu(x)\dd\nu(y) \Big).
\end{equation}
By Fenchel duality, we have with $c(x,y)=\frac12 \Vert y-x\Vert_2^2$ that 
\begin{align}\label{eq:dualentropic}
\frac12 T_\lambda(\mu,\nu) = \max_{u,v} F_\lambda(u,v) 
\end{align}
where the maximum is over pairs of continuous real-valued functions on $\RR^d$, $(u,v)\in \mathcal{C}(\Xx)^2$. Sinkhorn's algorithm is alternate maximization on $u$ and $v$: it starts with $u_0,v_0 =0$ and defines,
\begin{align*}
    u_{k+1} &= u_k -\lambda \log \int_{\RR^d} \exp((u_k(\cdot)+v_k(y)-c(\cdot,y))/\lambda)\dd\nu(y), & v_{k+1}&=v_k &&\text{if $k$ is odd}\\
    v_{k+1} &= v_k - \lambda \log \int_{\RR^d} \exp((u_k(x)+v_k(\cdot)-c(x,\cdot))/\lambda)\dd\mu(x), & u_{k+1}&=u_k &&\text{if $k$ is even}.
\end{align*}
This form of the iterations that distinguishes between even and odd updates is convenient for the analysis, but beware that the index  $k$ here is twice the index appearing in Proposition~\ref{prop:sinkhornscomputational}, so the statements are adjusted consequently.
We also introduce $\gamma_k = \exp((u_k(x)+v_k(y)-c(x,y))/\lambda)\mu \otimes \nu$, which is such that the update can be written: $u_{k+1} = u_k +\lambda \log (\dd\mu/\dd \pi^1_\# \gamma_k)$ if $k$ odd and $v_{k+1} = v_k +\lambda \log (\dd\nu/\dd \pi^2_\# \gamma_k)$ if $k$ even, where $\pi^1_\#\gamma$ is the marginal of $\gamma$ on the first factor of $\RR^d\times \RR^d$ and $\pi^2_\#\gamma$ its marginal on the second. The following is a rearrangement of some intermediate results in~\cite{dvurechensky2018computational} in a simplified form which is sufficient to our purpose. 

\begin{proposition}\label{prop:convergencesinkhorn}
Assume $c\geq 0$ and let $\Vert c\Vert_\infty = \sup_{(x,y)\in \Xx^2} c(x,y)$. Sinkhorn's iterates satisfy, for $k\geq 1$,
$$
0 \leq \max_{u,v} F_\lambda(u,v) - F_\lambda(u_k,v_k)\leq \frac{2 \Vert c\Vert^2_\infty}{\lambda k}
$$
\end{proposition}

\begin{proof}
First, remark that the iterations are such that $\int \dd\gamma_k=1$ for $k\geq 1$, so it holds $F_\lambda(u,v) = \int u \dd\mu + \int v \dd\nu$ for $(u,v)=(u_k,v_k)$ and also for any maximizer $(u,v)=(u^*,v^*)$. The key of the proof is the following equality first noticed by~\cite{altschuler2017near}. If $k$ is odd, then
\begin{align*}
    F_\lambda(u_{k+1},v_{k+1}) - F_\lambda(u_k,v_k) &= -\lambda \int \log \Big(\int \exp((u_k(x)+v_k(y)-c(x,y))/\lambda)\dd\nu(y)\Big)\dd\mu(x)\\
    &=\lambda \int \log(\dd\mu/\dd \pi^1_\# \gamma_k)\dd\mu=\lambda H(\mu, \pi^1_\# \gamma_k).
\end{align*}
Let us define $\Delta_k = F_\lambda(u^*,v^*)- F_\lambda(u_k,v_k)\geq 0$. Using Pinsker's inequality and Lemma~\ref{lem:boundsinkhorn} it follows
$$
\Delta_{k} - \Delta_{k+1} \geq \frac{\lambda}{2} \Vert \mu -  \pi^1_\#\gamma_k\Vert_1^2 \geq \frac{\lambda}{2\Vert c\Vert^2_\infty} \Delta_k^2.
$$
We can similarly prove the same inequality for $k$ even. We conclude as in the usual proof of gradient descent for smooth functions~\cite[Thm. 2.1.14]{nesterov2013introductory}: by dividing by $\Delta_k\Delta_{k+1}$ we have
$$
\frac{1}{\Delta_{k+1}} - \frac{1}{\Delta_k} \geq \frac{\lambda}{2\Vert c\Vert^2_\infty}\frac{\Delta_k}{\Delta_{k+1}} \geq \frac{\lambda}{2\Vert c\Vert^2_\infty}.
$$
Summing these inequalities yields a telescopic sum and we get $1/\Delta_k \geq \lambda k / (2\Vert c\Vert^2_\infty)$ which allows to conclude.
\end{proof}

From this analysis, we deduce the following complexity to approximate $T_\lambda$ and $T_0$ using Sinkhorn's iterations, adapted from~\cite{dvurechensky2018computational}.
\begin{proposition}\label{prop:computationalcomplexity}
Assume that $\mu_n = \sum_{i=1}^n p_i \delta_{x_i}$ and $\nu_n = \sum_{j=1}^n q_j \delta_{y_j}$ are discrete measures with $n$ atoms such that $p_i, q_j \geq \alpha/n$ for some $\alpha>0$. Then Sinkhorn's algorithm returns an $\epsilon$-accurate estimation of $T_\lambda(\mu,\nu)$ in time $O(n^2 \Vert c\Vert_\infty^2 /(\lambda \epsilon))$. Moreover, fixing $\lambda = \epsilon/4(\log(n)+\log(1/\alpha))$, it returns an $\epsilon$-accurate estimation of $T_0(\mu,\nu)$ in $O(n^2\log(n) \Vert c\Vert_\infty^2/\epsilon^2)$ operations.
\end{proposition}

\begin{proof}
The first claim is a direct consequence of Proposition~\ref{prop:convergencesinkhorn} since when $\mu$ and $\nu$ have a finite support of size $n$, an iteration of Sinkhorn can be performed with $O(n^2)$ operations. The second claim follows from the bound
$$
0\leq T_\lambda(\mu,\nu) - T_0(\mu,\nu) \leq 2\lambda H(\gamma^*, \mu\otimes \nu) \leq 4 \lambda (\log n +\log(1/\alpha))
$$
where $\gamma^*$ is the optimal transport plan for $T_0$.
\end{proof}

\begin{lemma}\label{lem:boundsinkhorn}
Under the assumptions and notations of Proposition~\ref{prop:convergencesinkhorn} it holds 
$$
\Delta_k \leq \Vert c\Vert_\infty \Big( \Vert \mu -  \pi^1_\#\gamma_k\Vert_1 + \Vert \nu -  \pi^2_\#\gamma_k\Vert_1\Big)
$$
where $\Vert \mu\Vert_1 \eqdef \sup_{\Vert u\Vert_\infty \leq 1}\int u(x)\dd\mu(x)$ denotes the total variation norm in the space of measures.
\end{lemma}
\begin{proof}
Remark that $F_\lambda$ is differentiable in $(u,v)$ with gradient $(\mu - \pi^1_\# \gamma_k, \nu - \pi^2_\# \gamma_k)$ at $(u_k,v_k)$. The concavity inequality then gives
$$
\Delta_k \leq \int (u^*-u_k)\dd(\mu - \pi^1_\# \gamma_k) + \int (v^*-v_k)\dd(\nu - \pi^2_\# \gamma_k).
$$
Also, for any $u\in \mathcal{C}(\Xx)$ and $\alpha = (\max u + \min u)/2$, using the fact that $\int \mu = \int \pi^1_\# \gamma_k$, we have
$$
\int u\dd(\mu - \pi^1_\# \gamma_k) = \int (u-\alpha)\dd(\mu - \pi^1_\# \gamma_k) \leq \frac12 (\max u - \min u) \Vert\mu - \pi^1_\# \gamma_k\Vert_1.
$$
Finally, for $u=u^*$ or $u=u_k$ for $k\geq 1$, we have, for some $v\in \mathcal{C}(\Xx)$, and for all $x,x'\in \Xx$
$$
u(x) = -\lambda \log \int \exp((v(y)-c(x,y))/\lambda)\dd\nu(y) \leq \Vert c\Vert_\infty - u(x')
$$
because $c(x,y)\leq c(x',y)+ \Vert c\Vert_\infty$. Thus $(\max u -\min u)/2 \leq \Vert c\Vert_\infty/2$. The conclusion follows by bounding all terms this way.
\end{proof}

\section{Properties of the plug-in estimator}\label{app:plugin}
In this section we prove Theorem~\ref{thm:plug-in} about the rate of convergence of $T_0(\hat \mu_n,\hat \nu_n)$ to $T_0(\mu,\nu)$ (we recall that, by definition $W_2^2(\mu,\nu)=T_0(\mu,\nu)$). We start with the following lemma which bounds the estimation error by simpler quantities. Note that we consider measures on the centered ball of radius $R$ in $\RR^d$, for some $R>0$, which is without loss of generality compared to other bounded sets since $T_\lambda(\mu,\nu)$ is invariant by translation of both measures. In the following lemma $\mu_n,\nu_n \in \Pp(\RR^d)$ can be unrelated to $\mu,\nu$ but this lemma will later be applied to the case where $\mu_n,\nu_n$ are empirical distributions of $n$ samples, hence our choice of notation.
\begin{lemma}\label{lem:plugin_decomposition}
Let $\mu,\nu, \mu_n,\nu_n \in \Pp(\RR^d)$ be concentrated on the centered ball of radius $R$. Then it holds
\begin{align*}
\Big\vert \frac12 T_0(\mu,\nu) - \frac12 T_0(\mu_n,\nu_n)\Big\vert 
&\leq
\Big\vert \frac12 \int \Vert x\Vert_2^2 \dd (\mu-\mu_n)(x)\Big\vert + \Big\vert \frac12 \int \Vert x\Vert_2^2 \dd (\nu-\nu_n)(x)\Big\vert \\
&+
\sup_{\varphi \in \mathcal{F}_R} \Big\vert \int \varphi \dd(\mu_n -\mu)\Big\vert + \sup_{\varphi \in \mathcal{F}_R} \Big\vert \int \varphi \dd(\nu_n -\nu)\Big\vert
\end{align*}
where $\mathcal{F}_R$ is the set of convex and $R$-Lipschitz functions on the ball of radius $R$.
\end{lemma}
\begin{proof}
The first part of the proof is fairly classical. By Kantorovich duality, we have
\begin{align*}
 \frac12 T_0(\mu,\nu) &= \max_{u,v\in \mathcal{C}(\Xx)} \int u \dd\mu + \int v \dd\nu
\end{align*}
where $\Xx$ is the closed ball of radius $R$ and under the constraint that $u(x)+v(y)\leq \frac12 \Vert y-x\Vert_2^2$ for all $(x,y) \in \Xx^2$ and there exists a maximizer~\cite{santambrogio2015optimal}. By expanding the square and changing the unknown $(\varphi,\psi)=(\frac12 \Vert \cdot \Vert_2^2 -u,\frac12 \Vert \cdot \Vert_2^2 -v )$, we can equivalently write
\begin{align*}
\frac12 T_0(\mu,\nu) &= \frac12 \int \Vert x\Vert_2^2 \dd \mu(x) +\frac12 \int \Vert x\Vert_2^2\dd\nu(x) - \min_{\varphi,\psi\in \mathcal{C}(\Xx)} \left( \int \varphi \dd\mu + \int \psi \dd\nu\right)
\end{align*}
under the constraint that $\varphi(x)+\psi(y)\geq \langle x, y\rangle$ for all $(x,y)\in \Xx^2$. In the minimization problem, fix an arbitrary $\psi\in \mathcal{C}(\Xx)$ and notice that the value of the objective cannot increase if we replace $\varphi$ by $\psi^*$ defined by $\psi^*(x) = \max_{y\in \Xx} \langle x,y\rangle - \psi(y)$ and the couple $(\psi^*,\psi)$ still satisfies the constraint. Repeating this process by now fixing $\psi^*$, we find that the couple $(\psi^{*},\psi^{**})$ satisfies the constraint and has a smaller objective value. 
Now, as a supremum of affine functions, $\psi^*$ is convex. For any $y_0\in \Xx$, let $x_0$ be such that $\psi^*(y_0) = \langle x_0,y_0\rangle -\psi(x_0)$, and observe that for all $y\in \Xx$
\begin{align*}
\left\{
    \begin{aligned}
    \psi^*(y_0) &= \langle x_0, y_0\rangle - \psi(y_0)\\
    \psi^*(y) &\geq \langle x_0, y\rangle - \psi(y)
    \end{aligned}
    \right. 
    &&\Rightarrow &&
    \psi^*(y_0) - \psi^*(y) \leq \langle x_0, y_0-y\rangle  \leq R\Vert y_0-y\Vert_2.
\end{align*}
Since $y_0$ and $y$ are arbitrary, this shows that $\psi^*$ is $R$-Lipschitz, i.e., $\vert \psi^*(y)-\psi^*(y')\vert \leq R\Vert y-y'\Vert_2$ for all $(y,y')\in \Xx^2$. We thus have 
$$
\frac12 T_0(\mu,\nu) = \frac12 \int \Vert x\Vert_2^2 \dd \mu(x) +\frac12 \int \Vert x\Vert_2^2\dd\nu(x) - \min_{\varphi \in \mathcal{F_R}} \Big( \int \varphi \dd\mu + \int \varphi^* \dd\nu\Big).
$$
The rest of the proof is inspired by~\cite[Prop. 2]{mena2019statistical} (which analyzes the sample complexity of $T_\lambda$ for $\lambda>0$). Let us denote $\mathcal{S}_{\mu,\nu}(\varphi) \eqdef \int \varphi \dd\mu +\int \varphi^*\dd\nu$ and $\varphi_{\mu,\nu}$ the minimizer of $\mathcal{S}_{\mu,\nu}$ over $\mathcal{F}_R$. By optimality, we have
$$
\mathcal{S}_{\mu,\nu}(\varphi_{\mu_n,\nu}) - \mathcal{S}_{\mu_n,\nu}(\varphi_{\mu_n,\nu}) \leq \mathcal{S}_{\mu,\nu}(\varphi_{\mu,\nu}) - \mathcal{S}_{\mu_n,\nu}(\varphi_{\mu_n,\nu}) \leq\mathcal{S}_{\mu,\nu}(\varphi_{\mu,\nu}) - \mathcal{S}_{\mu_n,\nu}(\varphi_{\mu,\nu}).
$$
It follows that 
\begin{align*}
\vert \mathcal{S}_{\mu,\nu}(\varphi_{\mu,\nu}) - \mathcal{S}_{\mu_n,\nu}(\varphi_{\mu_n,\nu})\vert 
&\leq 
\max \Big\{ \vert \mathcal{S}_{\mu,\nu}(\varphi_{\mu_n,\nu}) - \mathcal{S}_{\mu_n,\nu}(\varphi_{\mu_n,\nu})\vert , \vert \mathcal{S}_{\mu,\nu}(\varphi_{\mu,\nu}) - \mathcal{S}_{\mu_n,\nu}(\varphi_{\mu,\nu})\vert \Big\}\\
&\leq \sup_{\varphi \in \mathcal{F}_R}\vert \mathcal{S}_{\mu,\nu}(\varphi) - \mathcal{S}_{\mu_n,\nu}(\varphi)\vert = \sup_{\varphi \in \mathcal{F}_R} \Big\vert \int \varphi \dd(\mu_n -\mu)\Big\vert .
\end{align*}
As a consequence, we have
$$
\Big\vert \frac12 T_0(\mu,\nu) - \frac12 T_0(\mu_n,\nu)\Big\vert \leq \Big\vert \frac12 \int \Vert x\Vert_2^2 \dd (\mu-\mu_n)(x)\Big\vert + \sup_{\varphi \in \mathcal{F}_R} \Big\vert \int \varphi \dd(\mu_n -\mu)\Big\vert.
$$
We finally conclude with the triangle inequality
$$
\vert T_0(\mu,\nu) - T_0(\mu_n,\nu_n)\vert \leq \vert T_0(\mu,\nu) - T_0(\mu_n,\nu)\vert + \vert T_0(\mu_n,\nu) - T_0(\mu_n,\nu_n)\vert
$$
and by bounding the second term in the same fashion.
\end{proof}

The next technical step is to bound the supremum of an empirical process that appears in the bound of Lemma~\ref{lem:plugin_decomposition}.
\begin{lemma}\label{lemma:chaining_bound}
Let $\mu\in \Pp(\RR^d)$ be concentrated on the ball of radius $R$ and $\hat \mu_n$ an empirical distribution of $n$ independent samples. Then it holds
$$
\mathbf{E} \Big [\sup_{\varphi \in \mathcal{F}_R} \Big\vert \int \varphi \dd(\hat \mu_n -\mu)\Big\vert \Big] \lesssim
\begin{cases}
R^2n^{-1/2} & \text{if $d<4$},\\
R^2n^{-1/2}\log(n) & \text{if $d=4$},\\
R^2 n^{-2/d} & \text{if $d>4$}
\end{cases}
$$
where the notation $\lesssim$ hides a constant depending only on $d$ and $\mathcal{F}_R$ is defined in Lemma~\ref{lem:plugin_decomposition}.
\end{lemma}
\begin{proof}
First notice that we can include in the definition of $\mathcal{F}_R$ the property that $\varphi(0)=0$ without changing the supremum. With this additional property, we in particular have that $\Vert \varphi\Vert_\infty \leq R^2$ for all $\varphi \in \mathcal{F}_R$. By a classical symmetrization argument~\cite[Thm.~4.10]{wainwright2019high}, we have
$$
\mathbf{E} \Big [\sup_{\varphi \in \mathcal{F}_R} \Big\vert \int \varphi \dd(\hat\mu_n -\mu)\Big\vert \Big] \leq 2 \underbrace{\Esp_{\sigma,X}\Big[ \sup_{\varphi \in \mathcal{F}_R}\Big\vert\frac1n \sum_{i=1}^n \sigma_i \varphi(X_i) \Big\vert \Big]}_{\mathcal{R}_n(\mathcal{F}_R,\mu)}
$$
where $\sigma_1,\dots,\sigma_n$ are independent Rademacher random variables taking the values $\{-1,+1\}$ with equal probability and $X_1,\dots,X_n$ are independent random variables with law $\mu$. This quantity $\mathcal{R}_n(\mathcal{F}_R,\mu)$ is the Rademacher complexity of $\mathcal{F}_R$ under the distribution $\mu$. It can be bounded by Dudley's chaining technique (see~\cite[Thm. 5.22]{wainwright2019high} and the associated discussion): it holds, for some universal constant $C>0$,
$$
\mathcal{R}_n(\mathcal{F}_R,\mu) \leq C \inf_{\delta>0} \Big( \delta + n^{-1/2}\int_{\delta}^{R^2} \sqrt{\log N_\infty(\mathcal{F}_R,u)}\dd u\Big)
$$
where $N_\infty(\mathcal{F}_R,u)$ is the covering number of the set $\mathcal{F}_R$ for the metric $\Vert \cdot \Vert_\infty$ at scale $u$. Then we use the covering number bound of Bronshtein~\cite{bronshtein1976varepsilon}, as reported in~\cite[Thm.~1]{guntuboyina2012} which states that there exists constants $C_1,C_2>0$ depending only on $d$ such that if $u/R^2 \leq C_1$ then
$$
\log N_\infty(\mathcal{F}_R,u) \leq C_2 (u/R^2)^{-d/2}.
$$
After a change of variable we thus have that 
$$
\mathcal{R}_n(\mathcal{F}_R,\mu) \lesssim R^2 \Big(\inf_{\delta>0} \delta + n^{-1/2}\int_\delta^1 u^{-d/4}\dd u\Big). 
$$
The claim follows by optimizing over $\delta$ which gives $\delta =0$ for $d<4$, $\delta = n^{-1/2}$ for $d=4$ and $\delta =n^{-2/d}$ for $d>4$.
\end{proof}

We are now in position to conclude the proof of Theorem~\ref{thm:plug-in}.
\begin{proof}[Proof of Theorem~\ref{thm:plug-in}]
Let us assume without loss of generality that $\mu,\nu$ are concentrated on the centered closed ball of radius $R$ in $\RR^d$ (which can be taken as $R=1/2$ under our assumptions, but let us continue with an arbitrary $R$ for explicitness of the proof). Given Lemma~\ref{lem:plugin_decomposition} and Lemma~\ref{lemma:chaining_bound}, it only remains to bound the quantity
$$
A \eqdef \Esp \Big\vert \frac12 \int \Vert x\Vert_2^2 \dd (\mu-\hat \mu_n)(x)\Big\vert
$$
and the corresponding quantity for $\nu$. Considering independent samples of the random variable $Y=\frac12 \Vert X\Vert_2^2$ where the law of $X$ is $\mu$, our goal is to bound $A = \Esp \vert \frac1n \sum_{i=1}^n Y_i -\Esp Y\vert$. By Chebyshev’s inequality and the fact that the variance of $Y$ is bounded by $R^4$, we have for all $t\geq 0$,
$$
\mathbf{P}\Big[ \big \vert \frac1n \sum_{i=1}^n Y_i - \Esp Y\big \vert \geq t \Big] \leq \min\{ 1, R^4/(nt^2)\}.
$$
Finally, by the integral representation of the expectation of a nonnegative random variable we have
$$
A = \int_0^\infty \mathbf{P}\Big[ \big \vert \frac1n \sum_{i=1}^n Y_i - \Esp Y\big \vert \geq t \Big]\dd t \leq \frac{R^2}{\sqrt{n}} + \int_{R^2n^{-1/2}}^\infty \frac{R^4}{nt^2}\dd t = 2R^2 n^{-1/2}
$$
which is sufficient to conclude. The concentration bound is proved separately in Proposition~\ref{prop:concentration_unified}.
\end{proof}


Let us now prove the concentration bound, which is a consequence of the bounded difference inequality. We give a unified proof for $T_\lambda$ and $T_0$ since the argument is similar. The result for $T_0$ was known~\cite{weed2019sharp} but we are not aware of a similar result for $\lambda>0$ (note that the concentration bound in~\cite{genevay2019sample} has an undesirable exponential dependency in $\lambda$ and the central limit theorem in~\cite{mena2019statistical} does not a priori gives the dependency in $\lambda$).
\begin{proposition}\label{prop:concentration_unified}
Assume that $\mu,\nu\in \Pp(\RR^d)$ are concentrated on a set of diameter $D$. It holds for all $t\geq 0$, $\lambda\geq 0$ and $n\geq 1$,
$$
\mathbf{P} \Big[ \big\vert T_\lambda(\mu_n,\nu_n) - \Esp [T_\lambda(\mu_n,\nu_n)]\big\vert \geq t \Big] \leq 2 \exp(-nt^2/D^4).
$$
\end{proposition}
\begin{proof}
As in~\cite{weed2019sharp}, we want to apply the bounded difference inequality but we study the stability of the primal problem (instead of the dual) in order to cover the regularized case painlessly. The empirical measures are of the form $\mu_n = \frac1n \sum_{i=1}^n \delta_{x_i}$ and $\nu_n = \frac1n \sum_{j=1}^n \delta_{y_j}$. Let $c\in \RR^{n\times n}$ be the cost matrix with entries $c_{i,j} = \frac12 \Vert x_i - y_j\Vert_2^2$. With those notations, it holds
$$
\frac12 T_\lambda(\mu_n,\nu_n) = \min \sum_{i,j} c_{i,j} P_{i,j} + \lambda \sum_{i,j} P_{i,j}\log(n^2P_{i,j})
$$
where the minimum is over matrices $P\in \RR^{n\times n}_+$ such that $P\mathbf{1}=\mathbf{1}/n$ and $P^\top \mathbf{1}=\mathbf{1}/n$ (i.e.~$nP$ is bistochastic). Let $P^*$ be a minimizer. Now let $\tilde \mu_n =\frac1n (\sum_{i=1}^{n-1} \delta_{x_i} + \delta_{\tilde x_n})$ for some $\tilde x_i$ in the same set of diameter $D$. This changes one row in the cost matrix, each entry in this row being changed by at most $D^2/2$. Thus using $P^*$ as a candidate in the minimization problem defining $T_\lambda(\tilde \mu_n,\nu_n)$ we get $T_\lambda(\tilde \mu_n,\nu_n) \leq T_\lambda( \mu_n,\nu_n) + \frac{D^2}{n}$. Interchanging the role of $\mu_n$ and $\tilde \mu_n$, we get the reverse inequality and thus
$$
\vert T_\lambda(\tilde \mu_n,\nu_n) - T_\lambda( \mu_n,\nu_n\vert \leq \frac{D^2}{n}.
$$
The same stability can be shown about perturbing $\nu_n$ by one sample. The proposition follows by applying the bounded difference inequality~\cite[Cor.~2.21]{wainwright2019high}, paying attention to the fact that the total number of samples is $2n$.
\end{proof}

Finally, let us give the details of the proof of Proposition~\ref{prop:plugincomputational}.
\begin{proof}[Proof of Proposition~\ref{prop:plugincomputational}]
By the concentration result we have that with probability $1-\delta$, $\vert W_2^2(\mu_n,\nu_n)-\Esp W_2^2(\mu_n,\nu_n)\vert \lesssim n^{-1/2}\sqrt{\log(2/\delta)}$. Let us break down the proof into three cases depending on the dimension $d$.

If $d< 4$, then by choosing $n \gtrsim \log(2/\delta)\epsilon^{-2}$, the quantity $W_2^2(\mu_n,\nu_n)$ has the desired accuracy with probability $1-\delta$. Also choosing $\lambda \lesssim \epsilon/(2\log n)$ guarantees that $\vert \hat T_{\lambda,n} - \hat W_2^2\vert \lesssim \epsilon$. Thus, the computational complexity is $O(n^2/(\lambda\epsilon)) = \tilde O(\epsilon^{-6})$.

If $d>4$, we can choose $n\gtrsim \log(2/\delta)^{d/4}\epsilon^{-d/2}$ to reach the desired accuracy, which leads to a computational complexity in $\tilde O(\epsilon^{-d-2})$.

Finally if $d=4$, we can choose $n$ such that $\epsilon \asymp n^{-1/2}(\log(n) +\sqrt{\log(2/\delta)})$ which leads to a computational complexity in $O(n^2\log(n)\epsilon^{-2}) = O(\epsilon^{-6}(\log n + \sqrt{\log(2/\delta)})^4) = \tilde O(\epsilon^{-6})$.
\end{proof}

\section{Analysis of the Sinkhorn divergence estimator given samples}\label{app:sinkhorndivergencesstat}
Let us first state a result on the sample complexity to estimate $S_\lambda$ with $\hat S_{\lambda,n}$ which is defined, given $x_1,\dots,x_n$ i.i.d.~samples from $\mu$ and $y_1,\dots,y_n$ i.i.d.~samples from $\nu$, as $\hat S_{\lambda,n} = S_\lambda(\hat \mu_n,\hat \nu_n)$ as in Eq.~\eqref{eq:sinkhorndivergence} where $\hat \mu_n = \frac1n \sum_{i=1}^n \delta_{x_i}$ and $\hat \nu_n = \frac1n \sum_{i=1}^n \delta_{y_i}$.
Since the following result has not yet been stated in the precise form that we use, we give a short proof below. It essentially just requires to combine the results from~\cite{mena2019statistical} and~\cite{genevay2019sample}.
\begin{lemma}\label{lem:samplecomplexityentropy}
Let $\mu,\nu\in \Pp(\RR^d)$ be concentrated on a set of diameter $1$, let $\hat \mu_n,\hat \nu_n$ be empirical distributions with $n$ independent samples and let $d'=2\lfloor d/2\rfloor$. Then
$$
\Esp\Big[ \vert \hat S_{\lambda,n} -S_\lambda(\mu,\nu) \vert \Big] \lesssim (1+\lambda^{-d'/2}) n^{-1/2}
$$
where $\lesssim$ hides a constant that only depends on $d$.
\end{lemma}

\begin{proof}
It has been shown in~\cite[Cor.~2]{mena2019statistical}, with a strategy similar to that employed in the end of the proof of Lemma~\ref{lem:plugin_decomposition}, that
$$
\Big\vert\frac12 T_\lambda(\hat \mu_n, \hat\nu_n) -\frac12 T_\lambda(\mu,\nu)\Big\vert \leq \sup_{f \in \mathcal{F}} \Big\vert \int f \dd(\hat \mu_n-\mu)\Big\vert + \sup_{f \in \mathcal{F}} \Big\vert \int f \dd(\hat \nu_n-\nu)\Big\vert
$$
where $\mathcal{F}$ is any class of functions that is large enough to contain all the solutions to Eq.~\eqref{eq:dualentropic} for all pairs of measures $\mu,\nu \in \Pp(\RR^d)$ concentrated on a set of diameter $1$. It was shown in~\cite[Thm.~2]{genevay2019sample} that $\mathcal{F}$ can be chosen as a ball in the Sobolev space $H^{s}$, $s\geq 1$ with diameter $C(1+\lambda^{1-s})$ for some $C>0$ that only depends on $d$ and $s$. In particular, for $s= d'/2+1$, $H^{s}$ is a reproducible kernel Hilbert space. Thus, using the notion of Rademacher complexity introduced in the proof of Lemma~\ref{lemma:chaining_bound} and its bound for balls in reproducible kernel Hilbert spaces (as in~\cite[Prop.~2]{genevay2019sample}), it follows
$$
\Esp\Big[ \sup_{f\in \mathcal{F}} \Big\vert \int f\dd(\hat \mu_n-\mu)\Big\vert \Big] \leq 2\mathcal{R}_n(\mathcal{F},\mu) \lesssim (1+\lambda^{-d'/2})n^{-1/2}.
$$
This is sufficient to bound the expected estimation error of $T_\lambda$. Let us now turn our attention to $\hat S_{\lambda,n}$. It holds
\begin{multline*}
\vert \hat S_{\lambda,n} - S_\lambda(\mu,\nu)\vert \leq \vert T_\lambda(\hat \mu_n,\hat \nu_n) -T_\lambda(\mu,\nu)\vert \\
+\frac12 \vert T_\lambda(\hat \mu_{n},\hat \mu_{n}) - T_\lambda(\mu,\mu)\vert +\frac12 \vert T_\lambda(\hat \nu_{n},\hat \nu_{n})-T_\lambda(\nu,\nu)\vert. 
\end{multline*}
The argument in~\cite{mena2019statistical} goes through for each term and it follows that $\hat S_{\lambda,n}$ admits the same statistical bound (up to a constant) than $\hat T_{\lambda,n}$.
\end{proof}

\begin{proof}[Proof of Proposition~\ref{prop:samplesinkhorn}]
Let $W_2^2 = W_2^2(\mu,\nu)$ and $S_\lambda=S_\lambda(\mu,\nu)$. We consider the following error decomposition:
$$
\Esp \big[ \vert \hat S_{\lambda,n} -  W_2^2\vert\big]\leq \Esp \big[ \vert \hat S_{\lambda,n} -  S_\lambda \vert\big] + \vert S_\lambda - W_2^2\vert \lesssim (1+\lambda^{-d'/2}) n^{-1/2} + \lambda^2
$$
where the first bound is from Lemma~\ref{lem:samplecomplexityentropy} and the second bound is an assumption. We then optimize the bound in $\lambda$ which gives $\lambda \asymp n^{-1/(d'+4)}$ and an error in $n^{-2/(d'+4)}$. For the concentration bound, we use the argument in the proof of Proposition~\ref{prop:concentration_unified} in Appendix~\ref{app:plugin}. Observe that if only one of the samples drawn from $\mu$ is changed, the resulting change in $\hat S_{\lambda,n}$ is at most $2/n$ which leads to, by the bounded difference inequality,
$$
\mathbf{P} \Big[ \big\vert S_\lambda(\mu_n,\nu_n) - \Esp [S_\lambda(\mu_n,\nu_n)]\big\vert \geq t \Big] \leq 2 \exp(-nt^2/4).
$$
\end{proof}

\begin{proof}[Proof of Proposition~\ref{prop:timecomplexitySinkhorn}]
For $\hat S^{(k)}_{\lambda,n}$ we consider the error decomposition
$$
\vert \hat S^{(k)}_{\lambda,n} - W_2^2\vert \leq \vert  \hat S^{(k)}_{\lambda,n} -  \hat S_{\lambda,n}\vert +
\vert \hat S_{\lambda,n} - \Esp [\hat S_{\lambda,n}] \vert + 
\Esp \vert \hat S_{\lambda,n} - S_{\lambda}\vert +
\vert S_{\lambda} - W_2^2\vert.
$$
Let us choose $\lambda \asymp n^{-1/(d'+4)}$ as in Proposition~\ref{prop:samplesinkhorn}. By the concentration result of Proposition~\ref{prop:samplesinkhorn}, we have that with probability $1-\delta$, $\vert \hat S_{\lambda,n}-\Esp \hat S_{\lambda,n}\vert \lesssim n^{-1/2}\sqrt{\log(2/\delta)}$ and thus $\vert \hat S_{\lambda,n} - W_2^2\vert \lesssim n^{-2/(d'+4)} + n^{-1/2}\sqrt{\log(2/\delta)}$.
Thus by choosing $n \gtrsim \log(2/\delta)\epsilon^{-(d'+4)/2}$ the quantity $\hat S_{\lambda,n}$ has the desired accuracy with probability $1-\delta$.  It follows that the computational complexity is $O(n^2/(\lambda\epsilon)) = \tilde O(\epsilon^{-d'-5.5})$.

For the second claim, we just remark that $n^{-2/(d'+4)}$ dominates the rate of the plug-in estimator given in Theorem~\ref{thm:plug-in} for all $d$, so both estimators can achieve an error of this order. However the difference is that with $\hat S_{\lambda,n}$ a regularization level $\lambda \asymp \epsilon^{-1/2}$ is sufficient while $\lambda \lesssim \epsilon/\log(n)$ is required for $\hat T_{\lambda,n}$ to achieve this error $\epsilon$. The time complexity bounds then follows by Proposition~\ref{prop:sinkhornscomputational}.
\end{proof}

\section{Analysis of deterministic discretization}\label{app:grid}
In this section, we consider probability distributions on the torus $\mu,\nu \in \Pp(\TT^d)$ with densities with respect to the Lebesgue measure (also denoted $\mu,\nu$) and $c(x,y)=\frac12\Vert [y-x]\Vert_2^2$ which is half the squared distance on the torus. We denote $[x] = x+k_0$ where $k_0\in \mathbb{Z}^d$ is such that $\Vert x+k\Vert_2$ is minimal ($k_0$ is unique Lebesgue almost everywhere). We denote by $(u_\lambda,v_\lambda)$ the couple of minimizers of Eq.~\eqref{eq:dualentropyregularized} that are fixed points of Sinkhorn's iterations
\begin{align}\label{eq:fixedpointsinkhorn}
    u_\lambda(x) = -\lambda \log \int e^{(v_\lambda(y)-c(x,y))/\lambda}\dd\nu(y), && v_\lambda(y)=  - \lambda \log \int e^{(u_\lambda(x)-c(x,y))/\lambda}\dd\mu(x)
\end{align}
and such that $u_\lambda(0)=0$. These properties uniquely define $(u_\lambda,v_\lambda$) and we consider $p_\lambda(x,y) = \exp\big((u_\lambda(x)+v_\lambda(y)-c(x,y))/\lambda\big)\mu(x)\nu(y)$ which is the unique solution to~\eqref{eq:entropycost}. The following lemma gives some regularity estimates on $p_\lambda$. What is required in its proof is regularity of the marginals and of the cost function (which we fix to be the half squared-norm cost for consistency).

\begin{lemma}[Regularity of $p_\lambda$]\label{lem:regularityplan}
Assume that $\mu,\nu\in \Pp(\TT^d)$ admit $M$-Lipschitz continuous log-densities. Then for almost every $z \in (\TT^d)^2$ it holds
$$
 \Vert \nabla \log  p_\lambda(z)\Vert_2 \leq 4\sqrt{d}\lambda^{-1} + 2M.
$$
Moreover, it holds for all $z,z'\in (\TT^d)^2$
$$
\vert p_\lambda(z) -p_\lambda(z')\vert \leq (e^{(4\sqrt{d}\lambda^{-1}+M)\Vert [z-z']\Vert_2} - 1)p_\lambda(z). 
$$
\end{lemma}
\begin{proof}
By differentiating the definition of $p_\lambda$, we have for almost every $(x,y)\in (\TT^d)^2$
$$
\nabla_x \log p_\lambda(x,y) = \frac{1}{\lambda}(\nabla u_\lambda(x) -[x-y]) + \nabla m(x).
$$
where $m$ is the log-density of $\mu$. By differentiating Eq.~\eqref{eq:fixedpointsinkhorn}, we also have
$$
\nabla u_\lambda(x) = \int [x-y]e^{(u_\lambda(x)+v_\lambda(y)-c(x,y))/\lambda}\dd\nu(y)
$$
and thus $\Vert \nabla u_\lambda(x)\Vert_2 \leq \sup_{y\in \TT^d} \Vert [y-x]\Vert_2 =\sqrt{d}$.
It follows that $
\sup_{x,y \in \TT^d} \Vert \nabla_x\log  p_\lambda(x,y)\Vert_2 \leq \frac{2\sqrt{d}}{\lambda} + M$, from which we deduce the first bound by also taking into account the $\nabla_y$ component. Now let $\alpha =4\sqrt{d}\lambda^{-1}+2M$. By Gr\"onwall's inequality, we have
$e^{-\alpha \Vert[ z-z']\Vert_2}p_\lambda(z) \leq p_\lambda(z')\leq e^{\alpha \Vert [z-z']\Vert_2}p_\lambda(z)$ for all $z,z'\in (\TT^d)^2$. It follows that $\vert p_\lambda (z')-p_\lambda(z)\vert \leq \max\{ e^{\alpha \Vert [z'-z]\Vert_2}-1, 1-e^{-\alpha \Vert [z'-z]\Vert_2}\}p_\lambda(z)$ which implies our claim.
\end{proof}

For a measure $\mu \in \Pp(\TT^d)$ we call $\mu_h$ its finite volume discretization at resolution $h=1/m$ for $m\in \mathcal{N}$ on the grid $(\ZZ/m\ZZ)^d$. It is built via the following process: let $q_h:\TT^d\to \TT^d$ be defined by $q_h(x_1,\dots,x_d) = (\frac1m \lfloor mx_1+1/2\rfloor,\dots,\frac1m \lfloor mx_d +1/2\rfloor)$. It maps each point $x\in \TT^d$ to its closest point on the grid $(\mathbb{Z}/m\mathbb{Z})^d$ (with some arbitrary rule for ties). Then let $\mu_h \eqdef(q_h)_\#\mu$ which gives to each point in the grid the mass that $\mu$ gives to its surrounding cell. Also, let us label the points in $(\mathbb{Z}/m\mathbb{Z})^d$ from $1$ to $n=m^d$ as $(x_i)_{i=1}^n$ (we also use the notation $y_i=x_i$) and let us call $Q_j\subset \TT^d$ the set of points which are mapped to the point labeled by $j \in \{1,\dots,n\}$. We also call $Q_{i,j}=Q_i\times Q_j \subset (\TT^d)^2$.
We now state and prove a result that is slightly more precise than Proposition~\ref{prop:grid}  were we control the error made by replacing measures by their discretization in the estimation of $T_\lambda$.
\begin{proposition}[Stability under discretization]
Assume that $\mu,\nu \in \Pp(\TT^d)$ admit $M$-Lipschitz continuous log-densities and let $C>0$ be any constant. If  $h(M+\lambda^{-1})\leq C$ then
\begin{align*}
-h^2(1+M) &\lesssim T_\lambda(\mu_h,\nu_h)-T_\lambda(\mu,\nu) \lesssim \min\{h, h^2(\lambda^{-1} +M+1)\}
\end{align*}
where $\lesssim$ hides constants that only depend on $d$ and $C$.
\end{proposition}
\begin{proof}
The principle of the proof is to build admissible transport plans for the continuous (resp.~discrete) problem from an admissible transport plan for the discrete (resp.~continuous) problem and to bound the associated primal objectives functions.

\textbf{From discrete to continuous plans.}
Consider any $\gamma_h \in \Pi(\mu_h,\nu_h)$ and consider $\gamma \in \Pi(\mu,\nu)$ the (unique) measure with a constant density with respect to $\mu\otimes \nu$ on each cell $Q_{i,j}$ and such that $(q_h\otimes q_h)_\# \gamma =\gamma_h$ (see~\cite[Def.~1]{genevay2019sample} for a detailed construction in $\RR^d$). By construction, it holds $H(\gamma, \mu \otimes \nu) = H(\gamma_h, \mu_h\otimes \nu_h)$. Let us bound the difference $\Delta_{i,j} = \int_{Q_{i,j}} (\frac12 \Vert [y-x]\Vert_2^2-\frac12 \Vert [x_i-y_j]\Vert_2^2)\dd\gamma(x,y)$. For clarity, let us assume that $[x-y]=x-y$ for all $(x,y) \in Q_{i,j}$, the argument being the same in each cell. We start with a second order Taylor expansion of the cost (which is exact with our quadratic cost):
\begin{multline*}
    \frac12 \Vert y-x\Vert_2^2 - \frac12 \Vert x_i-y_j\Vert_2^2 = (x_i-y_j)^\top (x-x_i) + (y_j-x_i)^\top (y-y_i) \\
    +\frac12 \Vert x-x_i\Vert_2^2 + \frac12 \Vert y-y_i\Vert_2^2 -(x-x_i)^\top (y-y_i).
\end{multline*}
Integrating the terms in the second row over $Q_{i,j}$, we get a quantity bounded by $dh^2/2$. For the terms in the first row, we see that we have to bound integrals of the form 
$\vert \sum_j \int_{Q_{i,j}} (x_i-y_j)^\top (x-x_i) \dd\gamma(x,y)\vert \leq \sqrt{d} \vert \int_{Q_i} (x-x_i)\mu(x)\dd x\vert $. So let us consider specifically the following integral:
\begin{align*}
\Delta_{i} &=\Big\vert \int_{Q_{i}} (x-x_i)\mu(x)\dd x\Big\vert \\
&= \Big\vert \int_{Q_{i}} (x-x_i)(\mu(x)-\vert Q_{i}\vert^{-1} \int_{Q_i} \mu(x')\dd x')\dd x\Big\vert\\
&\leq \sqrt{d}h \vert Q_{i}\vert^{-1} \int_{Q_i^2} \vert \mu(x)-\mu(x')\vert\dd x \dd x'
\end{align*}
where we used the fact that $x_i$ is the center of mass of $Q_i$ for the Lebesgue measure and we denoted $\vert Q_i\vert$ the Lebesgue measure of $Q_{i}$. Now, since $\log \mu$ is $M$-Lipschitz an application of Gr\"onwall's inequality as in the proof of Lemma~\ref{lem:regularityplan} shows that $\vert \mu(x)-\mu(x')\vert \leq (e^{M\Vert x-x'\Vert_2}-1)\mu(x)$. It thus follows that
$$
\Delta_i \leq \sqrt{d} h (e^{Mh\sqrt{d}}-1)\mu(Q_i) \lesssim M h^2 \mu(Q_i).
$$
Putting all the bounds together and summing over all cells $Q_{i,j}$ we get
$$
\int_{(\TT^d)^2}\Big(\frac12 \Vert [y-x]\Vert_2^2-\frac12 \Vert [x_i-y_j]\Vert_2^2\Big)\dd\gamma(x,y) \lesssim h^2(1+M).
$$
From this it follows that for $\lambda\geq 0$, we have $T_\lambda(\mu,\nu)-T_\lambda(\mu_h,\nu_h)\lesssim h^2(1+M)$.

\textbf{From continuous to discrete plans.} Consider any $\gamma \in \Pi(\mu,\nu)$ and consider its discretization $\gamma_h = (q_h\otimes q_h)_\#\gamma$. By the ``information processing inequality'', it holds $H(\gamma_h,\mu_h\otimes \nu_h)\leq H(\gamma,\mu\otimes \nu)$. Also, since the cost function is $\sqrt{d}$-Lipschitz on $\TT^d$, we have the naive discretization bound
$$
\Big\vert \int_{(\TT^d)^2} \frac12 \Vert x-y\Vert^2_2 \dd(\gamma-\gamma_h)(x,y)\Big\vert \lesssim h .
$$
This is sufficient to deduce that $T_\lambda(\mu_h,\nu_h)-T_\lambda(\mu,\nu)\lesssim h$ for all $\lambda\geq 0$. Let us see however that a finer discretization bound can be given when $\gamma$ is the optimal solution of the entropy regularized problem using the regularity shown in Lemma~\ref{lem:regularityplan}. We denote $z=(x,y)\in (\TT^d)^2$ and $z_{i,j}=(x_i,y_i)$ and we have, by decomposing the error into a first and second order term as in the first part of the proof,
\begin{align*}
\Big\vert \int_{(\TT^d)^2} \frac12 \Vert y-x\Vert_2^2\dd(\gamma(x,y)-\gamma_h(x,y)) \Big\vert
&=\Big\vert \sum_{i,j} \int_{Q_{i,j}} (\frac12 \Vert y-x\Vert_2^2-\frac12 \Vert x_i-y_j\Vert_2^2)\dd\gamma(x,y)\Big\vert\\
&\lesssim \sum_{i,j} \Big\vert \int_{Q_{i,j}} (z-z_{i,j}) p_\lambda(z)\dd z \Big\vert  + h^2.
\end{align*}
It remains to estimate the integral terms as can be done as in the first part of the proof by using the regularity of $\log p_\lambda$ given by Lemma~\ref{lem:regularityplan}
\begin{align*}
\Big\vert \int_{Q_{i,j}} (z-z_{i,j}) p_\lambda(z)\dd z \Big\vert 
&\lesssim h \vert Q_{i,j}\vert^{-1} \int_{Q_{i,j}} \int_{Q_{i,j}} \vert p_\lambda(z)-p_\lambda(z')\vert\dd z\dd z'\\
&\leq h  (e^{(4\sqrt{d}\lambda^{-1}+M)\sqrt{d}h} - 1)p_\lambda(Q_{i,j})\\
&\lesssim h^2(\lambda^{-1}+M)p_\lambda(Q_{i,j}) .
\end{align*}
The conclusion follows by summing over all cells $Q_{i,j}$.
\end{proof}

We now proceed to the proof of Proposition~\ref{prop:gridperf}. This proof would be immediate if we were working on $\RR^d$ by combining the stability of Proposition~\ref{prop:grid} with the approximation error of Theorem~\ref{th:bias}. However, our framework in this section is that of the torus, and has to be so because there is no compactly supported measures with continuous log-densities on $\RR^d$. In the setting of the torus, the equivalence from Eq.~\eqref{eq:EntropicDynamic} holds for a slightly different cost function built from the heat kernel on the torus, as proved in~\cite{gigli2018benamoubrenier} for general manifolds. This cost function is
$$
\tilde c_\lambda(x,y) = - \lambda \log \Big(\sum_{k\in \mathbb{Z}^d}\exp\Big( -\frac{1}{2\lambda}\Vert x-y-k\Vert_2^2\Big)\Big).
$$
Let $\tilde T_{\lambda}(\mu,\nu)$ be the entropy regularized optimal transport cost as defined in Eq.~\eqref{eq:entropycost} where the cost function $c(x,y)=\frac12 \Vert [x-y]\Vert_2^2$ is replaced by $\tilde c_\lambda$, and let $\tilde S_\lambda$ be the corresponding Sinkhorn divergence, as defined in Eq.~\eqref{eq:sinkhorndivergence}. A direct extension of Theorem~\ref{th:bias} then gives that if $\mu,\nu\in \Pp(\RR^d)$ have bounded densities and supports then
\begin{equation}\label{eq:biastorus}
\vert \tilde S_\lambda(\mu,\nu) - W_2^2(\mu,\nu)\vert \leq \frac{\lambda^2}{4}\max\{ 2I_0(\mu,\nu),I_0(\mu,\mu)+I_0(\nu,\nu) \}.
\end{equation}
In the next lemma, we control the error that is made when replacing $\tilde S_\lambda$ by $S_\lambda$, which is asymptotically exponentially small.

\begin{lemma}\label{lem:fromheattotorus}
Assume that $\mu,\nu \in \Pp(\TT^d)$ admit log-densities which are Lipschitz continuous. Then there exists $c_1,c_1',c_2>0$ such that
$$
0\leq T_\lambda(\mu,\nu) - \tilde T_\lambda(\mu,\nu)\leq c_1 e^{-c_2/\lambda}.
$$
In particular, we have $\vert \tilde S_\lambda(\mu,\nu) - S_\lambda(\mu,\nu)\vert \leq c_1'e^{-c_2/\lambda}$.
\end{lemma}
In contrast to the other statements in this paper, this one is purely asymptotic in the sense that the constants may depend on $\mu$ and $\nu$. This is due to a technical difficulty near the cut-locus where the convergence of $\tilde c_\lambda$ towards $c$ is only in $O(\lambda)$ which is too slow for our purposes. We can avoid this difficulty by exploiting the fact that the optimal transport map stays away from the cut locus and using the uniform convergence of the dual potentials $(u_\lambda,v_\lambda)$ towards $(u_0,v_0)$ but we are not aware of quantitative versions of these results.
\begin{proof}
The inequality $\tilde T_\lambda(\mu,\nu)\leq T_\lambda(\mu,\nu)$ is immediate since $\tilde c_\lambda\leq c$. The main difficulty is thus to prove the other bound. For this, let $(u_\lambda,v_\lambda)$ be the unique pair of maximizers of Eq.~\eqref{eq:dualentropic} such that $u_\lambda(0)=0$. As $\lambda\to0$, this pair converges uniformly to a couple of functions $(u_0,v_0)$ which is the unique solution to the unregularized dual problem such that $u_0(0)=0$, see e.g.~\cite{berman2017sinkhorn}. Letting $\tilde F_\lambda$ be the dual of the regularized problem Eq.~\eqref{eq:dualentropyregularized} where $c$ is replaced by $\tilde c_\lambda$, we have $\frac12 \tilde T_\lambda(\mu,\nu) = \sup \tilde F_\lambda(u,v)$ where the supremum is over pairs of continuous functions on the torus. Thus we have 
\begin{align*}
\frac12 T_\lambda(\mu,\nu)-\frac12 \tilde T_\lambda(\mu,\nu) &\leq F_\lambda(u_\lambda,v_\lambda)-\tilde F_\lambda(u_\lambda,v_\lambda) \\
&= \lambda \int_{(\TT^d)^2} e^{(u_\lambda(x)+ v_\lambda(y)-c(x,y))/\lambda}\big(e^{(c-\tilde c)/\lambda}-1\big)\dd \mu(x)\dd \nu(y).
\end{align*}
It remains to bound this integral and we will do so by dividing the domain $(\TT^d)^2$ into two sets. 

By the regularity theory of optimal transport on the torus~\cite{cordero1999transport}, we know that $u_0$ is continuously differentiable (note that our assumption on the regularity of $\mu$ and $\nu$ is indeed stronger than H\"older continuity). It follows by~\cite[Lem.~2.4]{berman2017sinkhorn} that the optimal transport map $T$ is continuous and its graph $G = \{(x,T(x)) \;;\; x\in \TT^d\}$ does not intersect the singular set $S$ of $(x,y) \mapsto \Vert [y-x]\Vert_2^2$, i.e.\ the set where this function is not differentiable. As both sets are compact, they are thus at a positive distance $2\delta>0$ from each other. Let $G_{\delta}$ be the closed set of points that are at a distance less than or equal to $\delta$ from $G$ (which is itself at a distance $\delta$ from $S$). Since in our context $G$ is precisely the set of points $(x,y)$ where $u_0(x)+v_0(y)=c(x,y)$ (see again~\cite[Lem.~2.4]{berman2017sinkhorn}), there exists $\alpha>0$ such that $u_0(x)+v_0(x)-c(x,y)\leq -2\alpha$ for all $(x,y)\in G_\delta^c = (\TT^d)^2\setminus G_\delta$. 

Let $(I)$ and $(II)$ be the value of the integral above on $G_\delta^c$ and $G_\delta$ respectively, so that $T_\lambda(\mu,\nu)-\tilde T_\lambda(\mu,\nu)\leq 2(I) + 2(II)$. On the one hand, by uniform convergence of the potentials, there exists $\lambda_0>0$ such that $\forall \lambda<\lambda_0$, $\Vert u_\lambda -u_0\Vert_\infty + \Vert v_\lambda -v_0\Vert_\infty \leq \alpha$ and thus $\forall \lambda\leq \lambda_0$,
$$
(I) \leq \lambda e^{-\alpha/\lambda}\vert e^{\Vert c-\tilde c_\lambda\Vert_\infty/\lambda} +1\vert =o(e^{-\alpha/(2\lambda)})
$$
because $\tilde c_\lambda$ converges uniformly to $c$ as $\lambda\to 0$. On the other hand
$$
(II) \leq \lambda \sup_{z\in G_\delta}( e^{(c(z)-\tilde c(z))/\lambda} -1) = \lambda \sup_{z\in G_\delta} \sum_{k\in \mathbb{Z}^d\setminus \{k_0(z)\}} e^{(\Vert z-k_0(z)\Vert_2^2 - \Vert z - k\Vert_2^2)/(2\lambda)}
$$
where $k_0(z)$ is such that $\Vert [z]\Vert_2 = \Vert z - k_0\Vert_2$ and is unique for $z\in G_\delta$. Letting $\beta = \inf_{z\in G_\delta, k\neq k_0(z)} \Vert z-k\Vert_2^2 - \Vert z-k_0(z)\Vert_2^2$, we have $\beta>0$ since $G_\delta$ is at a positive distance from the singular set $S$ and we have $(II)\lesssim \lambda e^{-\beta/(2\lambda)}$ because the series $\sum_{k\neq k_0} e^{(\beta + \Vert z-k_0\Vert_2^2-\Vert z-k\Vert_2^2)/(2\lambda)}$ is nonincreasing in $\lambda$ (notice that the exponent is nonpositive). Summing $(I)$ and $(II)$ leads to the result.
\end{proof}

We are finally in a position to prove Proposition~\ref{prop:gridperf}.
\begin{proof}[Proof of Proposition~\ref{prop:gridperf}] 
We decompose the error as
$$
\vert S_{\lambda}(\mu_h,\nu_h) - W_2^2(\mu,\nu)\vert \leq 
\vert S_{\lambda}(\mu_h,\nu_h) - S_{\lambda}(\mu,\nu)\vert + 
\vert S_{\lambda}(\mu,\nu) -\tilde S_{\lambda}(\mu,\nu)\vert + 
\vert \tilde S_{\lambda}(\mu,\nu) - W_2^2(\mu,\nu)\vert.
$$
The first term is in $O(h^2(\lambda^{-1}+M+1))$ by Proposition~\ref{prop:grid}. The second term is bounded by $c_1e^{-c_2/\lambda}$ by Lemma~\ref{lem:fromheattotorus}. The third term is bounded by $(\lambda^2/4)\max\{I_0(\mu,\nu),I_0(\mu,\mu)+I_0(\nu,\nu)\}$ as seen in Eq.~\eqref{eq:biastorus}, which is a variation of Theorem~\ref{th:bias}. Moreover, the assumption that $\mu$ and $\nu$ have $M$-Lipschitz continuous log-densities leads to the bound $I_0(\mu,\mu),I_0(\nu,\nu)\leq M^2$, which justifies why the statement of Proposition~\ref{prop:gridperf} does not requires specifically that these quantities be finite. Thus, we have
$$
\vert S_{\lambda}(\mu_h,\nu_h) - W_2^2(\mu,\nu)\vert \lesssim h^2\lambda^{-1} + \lambda^2.
$$
Minimizing in $\lambda$ suggests to take $\lambda=h^{2/3}$ and leads to an error bound in $O(h^{4/3})$. In terms of the accuracy $\epsilon$, we thus have $h\asymp \epsilon^{3/4}$ and $\lambda\asymp \epsilon^{1/2}$. The computational complexity bound follows by Proposition~\ref{prop:sinkhornscomputational} which gives a bound in $O(n^2\lambda^{-1}\epsilon^{-1})$ and the fact that $n=h^{-d}\asymp \epsilon^{-3d/4}$, hence a bound in $O(\epsilon^{-3d/2-3/2})$. 

For the computational complexity bound via $T_\lambda$, we use the error decomposition
$$
\vert T_\lambda(\mu_h,\nu_h)-W_2^2(\mu,\nu)\vert \leq 
\vert T_\lambda(\mu_h,\nu_h)-T_0(\mu_h,\nu_h)\vert + \vert T_0(\mu_h,\nu_h)-T_0(\mu,\nu)\vert
$$
where the first term is in $O(\lambda\log(n))$ and the second term is in $O(h)$ by Proposition~\ref{prop:grid}. Thus to reach an accuracy $\epsilon>0$, we may choose $h\asymp \epsilon$ and $\lambda\asymp \epsilon/\log(n)$ which leads to a time complexity in $\tilde O(\epsilon^{-2d-2})$.
\end{proof}

\section{Analysis of the Gaussian case}\label{app:gaussian}
Let $\mu =\mathcal{N}(a,A)$ and $\nu = \mathcal{N}(b,B)$ be Gaussian probability distributions with means $a,b \in \RR^d$ and positive definite covariances $A,B \in \RR^{d\times d}$. The following explicit formula for $T_\lambda$ is proven in~\cite{anonymous2020gaussian}:
$$
T_\lambda(\mu,\nu) =  \Vert a-b\Vert_2^2 +  \tr(A) +  \tr(B) - 2\tr(D^{AB}_\lambda) + {d\lambda}(1-\log(2\lambda)) +\lambda\log \det (2D^{AB}_\lambda +\lambda I)
$$
where $A^{1/2}$ denotes the unique positive definite square root of a positive definite matrix $A$ and $D^{AB}_\lambda = (A^{1/2}BA^{1/2}+\lambda^2 I/4)^{1/2}$ (notice that $A^{1/2}BA^{1/2}=M^\top M$ for $M=B^{1/2}A^{1/2}$ is positive definite). When $\lambda =0$, we recover the well known explicit formula (see e.g.~\cite{bhatia2019bures}):
$$
W_2^2(\mu,\nu) = \Vert a-b\Vert_2^2 +  \tr(A) +  \tr(B) - 2\tr(S). 
$$
where $S=(A^{1/2}BA^{1/2})^{1/2}$. Notice that this expression involves the squared Bures distance~\cite{bhatia2019bures} between positive definite matrices defined as $\bures^2(A,B) \eqdef \tr(A) +\tr(B) - 2\tr(S)$.

The expression above leads to the following formula for $\Delta = S_\lambda(\mu,\nu) - W_2^2(\mu,\nu)$:
\begin{multline*}
    \Delta = (\tr(D^{AA}_\lambda) - \tr(D^{AA}_0)) + (\tr(D^{BB}_\lambda) - \tr(D^{BB}_0)) - 2 (\tr(D^{AB}_\lambda) - \tr(D^{AB}_0)) \\
    + \frac{\lambda}{2} \big(2\log \det (2D^{AB}_\lambda +\lambda I)-  \log \det (2D^{AA}_\lambda +\lambda I) - \log \det (2D^{BB}_\lambda +\lambda I\big).
\end{multline*}
\paragraph{Fourth-order expansion of $\Delta$.} Let us first expand individual terms using the fact that all the matrices involved are positive definite. We have
\begin{align*}
 D_\lambda^{AA} &=
 A(I +(\lambda^2/4)A^{-2})^{1/2}\\
 &=A +\frac{\lambda^2}{8}A^{-1}-\frac{\lambda^4}{128}A^{-3} +O(\lambda^5).
\end{align*}
Also, since $\log \det (I+\lambda A) = \lambda \tr(A) -(\lambda^2/2) \tr(A^2) +(\lambda^3/3)\tr(A^3)+ O(\lambda^4)$, we  obtain the expansion
\begin{align*}
    \frac{\lambda}{2}\log\det (2D_\lambda^{AA}+\lambda I) &= \frac\lambda2 \log\det (2A +(\lambda^2/4)A^{-1} +\lambda I +O(\lambda^4))\\
    &=\frac\lambda2\log\det (2A) +\frac\lambda2\log\det(I +(\lambda/2)A^{-1} +(\lambda^2/8)A^{-2}+O(\lambda^4))\\
    &=\frac\lambda2\log\det (2A) +\frac{\lambda^2}{4}\tr(A^{-1}) - \frac{\lambda^4}{96}\tr(A^{-3}) + O(\lambda^5).
\end{align*}
Putting all pieces together with the notation $S = (A^{1/2}BA^{1/2})^{1/2}$ leads to
\begin{multline*}
\Delta = \frac{\lambda^2}{8} \tr(A^{-1}) -\frac{\lambda^4}{128}\tr(A^{-3}) + \frac{\lambda^2}{8} \tr(B^{-1}) -\frac{\lambda^4}{128}\tr(B^{-3})-\frac{\lambda^2}{4} \tr(S^{-1}) +\frac{\lambda^4}{64}\tr(S^{-3}) \\
+\lambda \log\det(2S) -\frac{\lambda}{2} \log\det(2A)  -\frac{\lambda}{2} \log\det(2B)
\\
+\frac{\lambda^2}{2} \tr(S^{-1}) 
-\frac{\lambda^2}{4} \tr(A^{-1}) 
 -\frac{\lambda^2}{4} \tr(B^{-1}) 
 -\frac{\lambda^4}{48}\tr(S^{-3})
 +\frac{\lambda^4}{96}\tr(A^{-3})
 +\frac{\lambda^4}{96}\tr(B^{-3})
 + O(\lambda^5).
\end{multline*}
The $\log\det$ terms cancel each other and some simplifications in the other terms lead to
$$
\Delta = \frac{\lambda^2}{8} \big(2 \tr(S^{-1}) - \tr(A^{-1}) - \tr(B^{-1})\big) - \frac{\lambda^4}{384}\big(2 \tr(S^{-3}) - \tr(A^{-3}) - \tr(B^{-3})\big)+ O(\lambda^5).
$$
Interestingly, this expression can be expressed purely in terms of Bures distances:
$$
S_\lambda(\mu,\nu) - W_2^2(\mu,\nu) = -\frac{\lambda^2}{8} \bures^2(A^{-1},B^{-1}) + \frac{\lambda^4}{384}\bures^2(A^{-3},B^{-3}) +O(\lambda^5).
$$
This shows that the terms in this expansion are non-zero unless $A=B$ and also determines their sign.

\section{Numerical settings and additional experiments}\label{app:additionalnumerics}

\subsection{Sampling method} In this paragraph, we detail the setting of the \emph{random sampling} experiments (Figure~\ref{fig:sampling} and Figure~\ref{fig:sampling_appendix}). In those experiments, the distributions $\mu$ and $\nu$ are elliptically contoured and centered, which allows to have a closed form expression for the optimal transport cost $T_0$ and the dual potential $\varphi$ (the Lagrange multiplier associated to the first marginal constraint in the computation of $T_0(\mu,\nu)$ in Eq.~\eqref{eq:entropycost}), which only depends on the two covariances~\cite{bhatia2019bures}. Specifically, given two measures $\mu,\nu$ that belong to the same family of elliptically contoured distributions, with respective covariances $A$ and $B$ and with $0$ means, we have
\begin{align*}
T_0(\mu,\nu) = \bures^2(A,B) && \text{and} && \varphi(x) = x^\top (\mathrm{Id}-M)x
\end{align*}
where $\bures^2(A,B) =\tr (A) +\tr(B) -2\tr(S)$ and $M = A^{1/2}SA^{1/2}$ where $S$ is as defined in Appendix~\ref{app:gaussian}. Let us detail how we have chosen the covariances and our choice of elliptically contoured distribution.

\paragraph{Choice of the covariances.} The covariances $A, B \in \RR^{d\times d}$ are generated  randomly, independently and identically according to the following process, that we detail for $A$. Let $M\in \RR^{d\times k}$ be a random matrix with i.i.d.~entries following a standard normal distribution $\mathcal{N}(0,1)$, with $k = d/\alpha$ for some $\alpha \in (0,1)$. We then define $\tilde A = MM^\top$, which is a random positive semidefinite matrix. By non-asymptotic versions of the Mar\v{c}enko-Pastur Theorem (e.g.~\cite[Eq.(1.11)]{wainwright2019high}), the eigenvalues of $\tilde A$ are contained within a small enlargement of the interval $[(1-\sqrt{\alpha})^2,(1+\sqrt{\alpha})^2]$ with a high probability that increases with $d$. We then define $A = \tilde A/\tr \tilde A$. With our choice $\alpha = 1/3$, this allows to define generic covariance matrices of trace $1$ with a controlled anisotropy: the ratio between the largest and smallest eigenvalue is with high probability of order $0.07$ for large $d$ (but note that since we work with relatively small values of $d$, this ratio is subject to  fluctuations).

\paragraph{Choice of the distributions.} Given a covariance $A$ we generate a sample $X$ as follows:
\begin{enumerate}
    \item $U\sim \mathcal{U}(\mathbb{S}^{d-1})$ ( $U$ is uniformly distributed on the sphere in $\RR^d$)
    \item $Z \sim \mathcal{N}(0,1)$
    \item $R = \alpha \vert  \arctan(Z/\beta)\vert^{1/d}$ where $\alpha>0$ is such that $\Esp [R^2] = d$
    \item $X = R\cdot A^{1/2}U$
\end{enumerate}
Here $\beta>0$ is a free parameter that determines the shape of the distribution and we have chosen $\beta =2$ because it tends to yield nice bell shaped densities (see Figure~\ref{fig:densityrandomsampling}). Also, $\alpha$ is a quantity that only depends on $d$ and $\beta$ that we estimate via Monte-Carlo integration. Let us describe the distribution of $X$.
\begin{proposition}\label{prop:densityexplicit}
The law of $X$ is elliptically contoured, centered, and has a compact support. Its covariance is $A$ and its density with respect to the Lebesgue measure (denoted by $\mu(x)$) is given by
\begin{equation}\label{eq:densityexplicit}
\mu(x) \propto (1+\tan(y)^2)\exp(-\beta^2\tan(y)^2/2)
\end{equation}
where $y = (\Vert x\Vert_{A^{-1}}/\alpha)^d$ and $\Vert x\Vert_{A^{-1}}^2 = x^\top A^{-1} x$. In particular, if $A$ is nonsingular then its Fisher information is finite: $I_0(\mu,\mu)<\infty$.
\end{proposition}
It follows that if $\mu$ and $\nu$ are the densities of random variables generated via this procedure, with respective covariances $A$ and $B$, then Theorem~\ref{th:bias} together with Proposition~\ref{cor:FirstFisherBound} guarantee that Proposition~\ref{prop:samplesinkhorn} applies. We illustrate the results of Proposition~\ref{prop:densityexplicit} in Figure~\ref{fig:densityrandomsampling}.

\begin{proof}
By construction $\mu$ is elliptically contoured and centered~\cite[Chap.~2]{fang2018symmetric}. It is compactly supported because the range of $z\mapsto \vert \arctan(z/\beta)\vert$ is $[0,\pi/2)$. Also the covariance of $X$ is 
$$
\Esp\big [XX^\top \big] = \frac1d \Esp\big[R^2\big] A =A.
$$
Let $Y =\arctan(Z/\beta)$ and let $F_Y$ (resp.~$f_Y$) be the cumulative (resp.~probability) distribution function of $Y$. We have for $x\in \mathbb{R}$,
$$
F_R(x) = \mathbf{P}\big[ R\leq x\big] = \mathbf{P}\big[ \alpha \vert \arctan(Z/\beta)\vert^{1/d} \leq x \big] =\mathbf{P}\big[ \vert Y\vert \leq (x/\alpha)^d \big] = F_{\vert Y\vert}((x/\alpha)^d).
$$
Differentiating this relation, it follows that
$
f_R(x)\propto x^{d-1} f_{\vert Y\vert}((x/\alpha)^d).
$
Then by~\cite[Thm.~2.9 \& Eq.~(2.43)]{fang2018symmetric}, we have
$$
\mu(x) \propto \Vert x\Vert^{1-d}_{A^{-1}} f_R(\Vert x\Vert_{A^{-1}}) \propto f_{\vert Y\vert}((\Vert x\Vert_{A^{-1}}/\alpha)^d).
$$
It thus remains to compute the density $f_{\vert Y\vert}$ which, by symmetry of $Y$ around $0$, is precisely twice the density $f_{Y}$ for nonnegative arguments. Denoting $g(z) = \arctan(z/\beta)$, by the change of variable formula, we have
$$
f_Y(y) = \frac{f_Z(g^{-1}(y))}{g'(g^{-1}(y))}\propto (1+\tan(y)^2)\cdot \exp(-\beta^2\tan(y)^2/2)
$$
which gives the density of $\mu$, up to a multiplicative constant. Let us now show that the Fisher information $I_0(\mu,\mu) = \int_{\RR^d} \Vert\frac{\nabla \mu(x)}{\mu(x)}\Vert^2_2\mu(x)\dd x$ is finite, with the assumption that $A=\mathrm{Id}$ for simplicity (the general case can be treated similarly). We have $\mu(x) = f_Y(h(\Vert x\Vert_2))$ with $h(r) = (r/\alpha)^d$ and by direct computations:
\begin{align*}
    I_0(\mu,\mu) &\propto \int_{\RR^d}\!\Big(\frac{f'(h(\Vert x\Vert_2))}{f(h(\Vert x\Vert_2))}\Big)^2 \Vert x\Vert^{2d-2}_2 f(h(\Vert x\Vert_2))\dd x 
     \propto \int_{0}^{\pi/2}\!\Big(\frac{f'(h(r))}{f(h(r))}\Big)^2 r^{3d-3} f(h(r))\dd r\\
    f'_Y(y) &\propto \exp(-\beta^2\tan(y)^2/2)\big(\beta^2\tan(y)(1+\tan(y)^2)(1-\beta^2\tan(y)^2)\big).
\end{align*}
Then by posing $z = \tan h(r)$,  we get
$$
I_0(\mu,\mu) \propto \int_{\RR_+} (\beta^4 z^2 (1-\beta^2z^2)^2 \arctan(z)^{2-2/d}) \exp(-\beta^2z^2/2)\dd z 
$$
where $\propto$ in those computations just means that the right-hand side is finite if and only if the left-hand side is finite. Since the right-hand side is finite, this shows that $I_0(\mu,\mu)<\infty$. 
\end{proof}
\begin{figure}
\centering
\begin{subfigure}{0.4\linewidth}
    \centering
    \includegraphics[scale = 0.4]{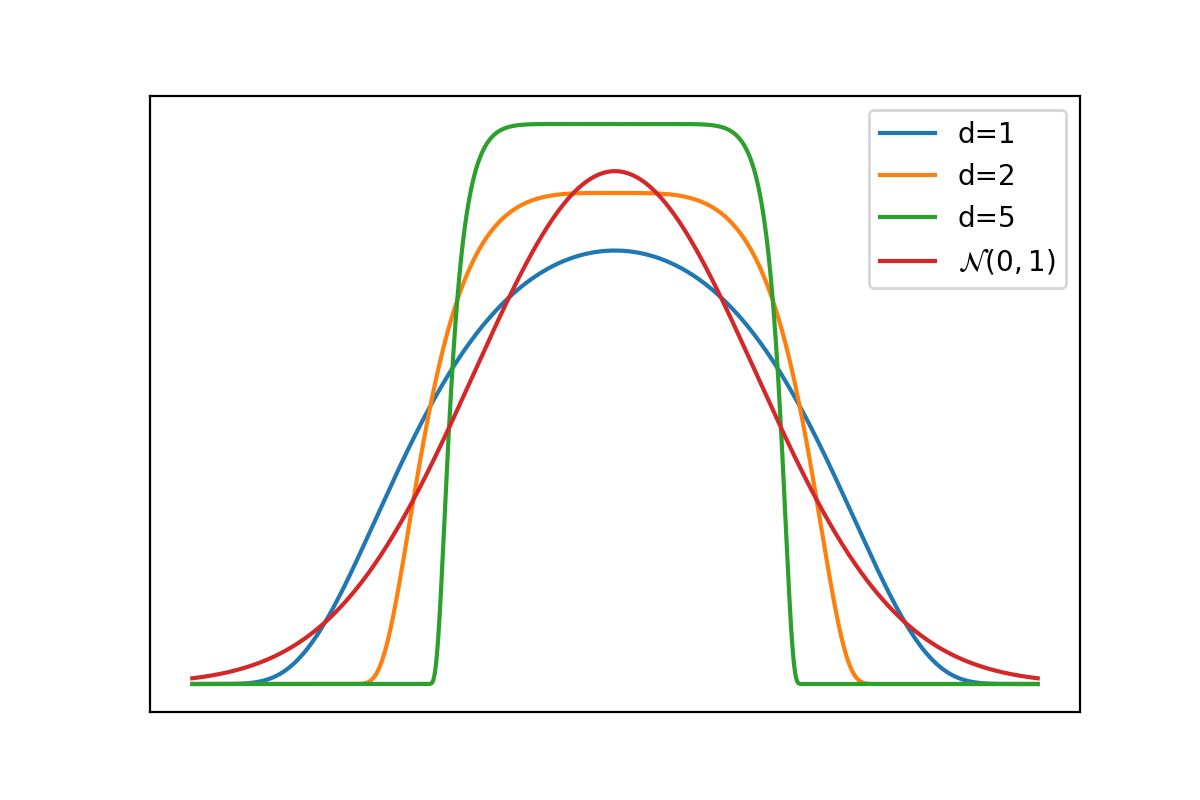}
\end{subfigure}\hspace{1cm}
\begin{subfigure}{0.4\linewidth}
    \centering
    \includegraphics[scale=0.16]{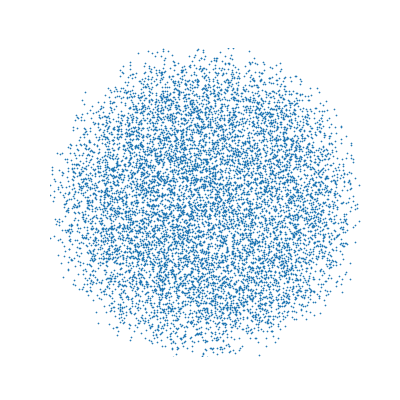}
\end{subfigure}
    \caption{Density used for the random sampling experiments, when $A=\mathrm{Id}/d$. Left: radial profile of the density as given by Eq.~\eqref{eq:densityexplicit}, i.e.~$t\mapsto \mu(t\vec{u})$ for some $\vec{u}\in \mathbb{S}^{d-1}$. Right: $10^4$ samples for $d=2$.}
    \label{fig:densityrandomsampling}
\end{figure}

\subsection{Additional random sampling experiment}
On Figure~\ref{fig:sampling_appendix}, we show the same experiment as in Section~\ref{sec:numerics} but in dimension $d=10$ and moreover we report the error on the transport cost $T_0(\mu,\nu)$ and the rate of Theorem~\ref{thm:plug-in}, which were not shown on Figure~\ref{fig:sampling}. The plot on the right shows the estimation error on $T_0(\mu,\nu)$, which is the quantity that we control in our theoretical analysis. This plot confirms several of our results: (i) the convergence rate in $n^{-2/d}$ of the plug-in estimator proved in Theorem~\ref{thm:plug-in} (note that we compute it with a small entropic regularization, which might explain the slight deviation from the rate $n^{-2/d}$ that we observe for $n$ large), and (ii) the fact that $T_\lambda$ has a much larger bias than $S_\lambda$ and $R_\lambda$. Even more interestingly, $S_\lambda$ and $R_\lambda$ have a smaller error than the plug-in estimator. However, we should also be cautious when interpreting such a plot because $T_0(\mu,\nu)$ is a scalar, and it is easy to make the error vanish when varying a parameter, such as $n$ or $\lambda$. In particular, the local minimum observed for $S_\lambda$ and $R_\lambda$ is simply due to the fact that the error changes its sign as $n$ grows.

This phenomenon led us to report the error on a different quantity, the $L^1$ error on the potential, which is not subject to this phenomenon and which also raises interesting open questions. Notice however that this quantity may behave quite differently than the estimation error on $T_0(\mu,\nu)$. In particular, we see on Figure~\ref{fig:densityrandomsampling}-(left), that the rate of convergence of the plug-in estimator is in fact faster than $n^{-2/d}$ in this experiment.

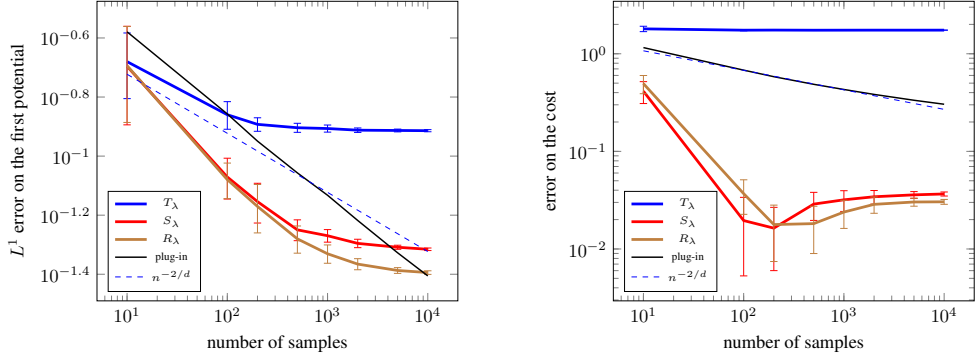
\begin{figure}[h]
\centering
\hbox to \linewidth{ \hss
  \begin{subfigure}[b]{0.5\textwidth}
  \centering
    \begin{tikzpicture}[scale = 0.7]
\begin{loglogaxis}[xlabel = number of samples, ylabel = $L^1$ error on the first potential, no markers, legend pos=south west,  legend style={font=\tiny}]
  \addplot[line width = 1.5pt, color = blue,  error bars/.cd, y dir=both, y explicit] table [x= Nsamples, y= error_on_potentials, y error = std_on_potentials, col sep=space] {figures/OT_lambda_1_dimension_10.csv};
  \addlegendentry{$T_\lambda$}
  \addplot[line width = 1.5pt, color = red,  error bars/.cd, y dir=both, y explicit] table [x= Nsamples, y= error_on_potentials, ,y error = std_on_potentials, col sep=space] {figures/S_lambda_1_dimension_10.csv};
  \addlegendentry{$S_\lambda$}
  \addplot[line width = 1.5pt, color = brown,  error bars/.cd, y dir=both, y explicit] table [x= Nsamples, y= error_on_potentials, y error = std_on_potentials, col sep=space] {figures/R1_lambda_1_dimension_10.csv};
  \addlegendentry{$R_\lambda$}
  \addplot[thick, black] table [x= Nsamples, y= error_on_potentials, col sep=space] {figures/OT_lambda_0.01_dimension_10.csv};
    \addlegendentry{plug-in}
            \addplot+[samples at = {10, 100, 200, 500, 1000, 2000, 5000, 10000}, dashed ] {0.3*(x)^(-2/10)};
            \addlegendentry{$n^{-2/d}$}
\end{loglogaxis}
\end{tikzpicture}
\end{subfigure}
\hss
\begin{subfigure}[b]{0.5\textwidth}
\centering
    \begin{tikzpicture}[scale = 0.7]
  \begin{loglogaxis}[xlabel = number of samples, ylabel = error on the cost, no markers, legend pos=south west,  legend style={font=\tiny}]
  \addplot[line width = 1.5pt, color = blue,  error bars/.cd, y dir=both, y explicit] table [x= Nsamples, y= error_on_cost,  y error = std_on_cost, col sep=space] {figures/OT_lambda_1_dimension_10.csv};
  \addlegendentry{$T_\lambda$}
  \addplot[line width = 1.5pt, color = red,  error bars/.cd, y dir=both, y explicit] table [x= Nsamples, y= error_on_cost, y error = std_on_cost, col sep=space] {figures/S_lambda_1_dimension_10.csv};
  \addlegendentry{$S_\lambda$}
  \addplot[line width = 1.5pt, color = brown,  error bars/.cd, y dir=both, y explicit] table [x= Nsamples, y= error_on_cost,  y error = std_on_cost, col sep=space] {figures/R1_lambda_1_dimension_10.csv};
  \addlegendentry{$R_\lambda$}
  \addplot[thick, black] table [x= Nsamples, y= error_on_cost, col sep=space] {figures/OT_lambda_0.01_dimension_10.csv};
    \addlegendentry{plug-in}
            \addplot+[samples at = {10, 100, 200, 500, 1000, 2000, 5000, 10000}, dashed ] {1.7*(x)^(-2/10)};
            \addlegendentry{$n^{-2/d}$}
\end{loglogaxis}
\end{tikzpicture}
\end{subfigure}
\hss
}
  \caption{$L^1$ error  on the first potential (left) and error on the estimated cost (right) for different estimators, for $\mu,\nu$ smooth compactly supported distributions with $d=10$, as a function of $n$ for $\lambda = 1$. Error bars show the standard deviation on 30 realizations}\label{fig:sampling_appendix}
\end{figure}

\subsection{Additional figures for the discretization experiment}
Figure~\ref{fig-deterministic-potentials} shows the same setting as  on Figure~\ref{fig-deterministic} and gives more details. The densities of $\mu$ and $\nu$ on the $1$-dimensional torus $\TT$ are shown on the top row at several levels of discretization. The two other rows show the evolution of the estimated potentials as $n$ varies for the optimal $\lambda$ (middle row) or as $\lambda$ varies for $n$ large (bottom row) towards the true potentials $(u_0,v_0)$ (shown in dark color). Here $u_0$ is the Lagrange multiplier associated to the first marginal constraint in the computation of $T_0(\mu,\nu)$ in Eq.~\eqref{eq:entropycost} and $v_0$ is the one associated to the second marginal constraint. On Figure~\ref{fig-deterministic-potentials}, we denote by $(u_h,v_h)$ the potentials associated to the estimator $T_\lambda$ and by $(\bar u_h,\bar v_h)$ those associated to the estimator $S_\lambda$, as defined in Section~\ref{sec:numerics}. This figure illustrates that for $\lambda$ large, the error is systematically smaller with the debiasing terms.

\begin{figure}[h]
		\centering
		\includegraphics[width=\linewidth]{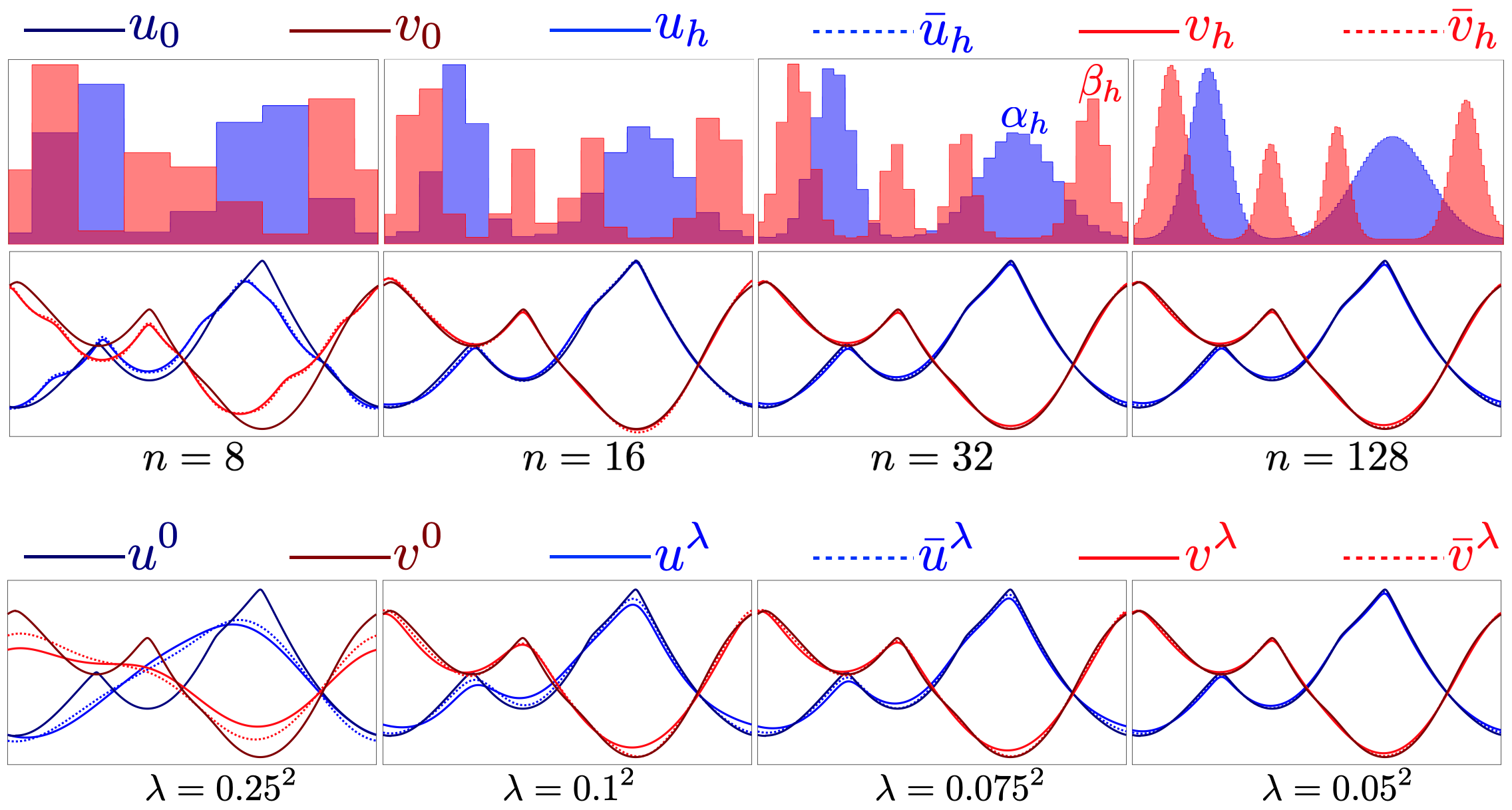}
	\caption{
	    Rows 1 and 2: convergence of the dual potentials $(u_{0,h},v_{0,h})$ and $(\bar u_{0,h},\bar v_{0,h})$ towards $(u_0,v_0)$ for decreasing sampling step $h$. The top row shows the discretized measures $(\mu_h,\nu_h)$ (the measure is a sum of Dirac masses, which is vizualized as a piecewise constant function to indicate the cells over which the densities have been integrated).   
	    Last row: same but for the convergence of
	    $(u_{\lambda,0},v_{\lambda,0})$ and $(\bar u_{\lambda,0},\bar v_{\lambda,0})$ as $\lambda$
	    gets smaller. 
	    }
	    \label{fig-deterministic-potentials}
\end{figure}